\documentclass[11pt]{amsart}
\usepackage{hyperref}
\usepackage{amsfonts}
\usepackage{amsmath}
\usepackage{amssymb}
\usepackage{amscd}
\usepackage{graphics}
\usepackage{latexsym}
\usepackage{enumitem}
\usepackage{color}
\usepackage{epsfig}
\usepackage{amsopn}
\usepackage{xypic}
\usepackage{mathrsfs}
\usepackage{array}

\usepackage{booktabs}
\usepackage{longtable}
\theoremstyle{plain}
\newtheorem{lemma}{Lemma}[section]
\newtheorem*{theorem*}{Theorem}
\newtheorem*{lemma*}{Lemma}
\newtheorem*{proposition*}{Proposition}
\newtheorem*{conjecture*}{Conjecture}
\newtheorem*{corollary*}{Corollary}
\newtheorem*{problem*}{Problem}
\newtheorem{theorem}[lemma]{Theorem}

\newtheorem{corollary}[lemma]{Corollary}
\newtheorem{proposition}[lemma]{Proposition}
\newtheorem{proposition-definition}[lemma]{Proposition and Definition}

\theoremstyle{definition}
\newtheorem{definition}[lemma]{Definition}

\newtheorem{remark}[lemma]{Remark}

\newtheorem{warning}[lemma]{Warning}

\newcommand{\F}[1]{\mathscr{#1}}

\newcommand{\C}{\mathbb{C}}
\newcommand{\Q}{\mathbb{Q}}
\newcommand{\R}{\mathbb{R}}

\newcommand{\OO}{\mathcal{O}}
\newcommand{\te}{\otimes}

\newcommand{\cF}{\mathcal F}
\newcommand{\cQ}{\mathcal Q}
\newcommand{\cA}{\mathcal A}

\renewcommand{\cD}{\mathcal{D}}
\newcommand{\rH}{\mbox{H}}
\newcommand{\gr}{\mathrm{gr}}
\newcommand{\cZ}{\mathcal{Z}}

\newcommand{\ZZ}{\mathbb{Z}}

\renewcommand{\P}{\mathbb{P}}
\newcommand{\PP}{\mathbb{P}}

\DeclareMathOperator{\ch}{ch}

\DeclareMathOperator{\Aut}{Aut}

\DeclareMathOperator{\Hom}{Hom}

\DeclareMathOperator{\Pic}{Pic}

\DeclareMathOperator{\rk}{rk}

\DeclareMathOperator{\Ext}{Ext}
\DeclareMathOperator{\ext}{ext}
\DeclareMathOperator{\Supp}{Supp}

\DeclareMathOperator{\im}{Im}
\DeclareMathOperator{\Flag}{Flag}

\DeclareMathOperator{\coh}{coh}

\DeclareMathOperator{\Amp}{Amp}

\DeclareMathOperator{\Nef}{Nef}

\DeclareMathOperator{\NS}{NS}

\begin{document}

\title{The ample cone of moduli spaces of sheaves on the plane}

\date{\today}
\author[I. Coskun]{Izzet Coskun}
\author[J. Huizenga]{Jack Huizenga}

\address{Department of Mathematics, Statistics and CS \\University of Illinois at Chicago, Chicago, IL 60607}
\email{coskun@math.uic.edu}
\email{huizenga@math.uic.edu}
\thanks{During the preparation of this article the first author was partially supported by the NSF CAREER grant DMS-0950951535,  and the second author was partially supported by a National Science Foundation Mathematical Sciences Postdoctoral Research Fellowship}

\begin{abstract}
Let $\xi$ be the Chern character of a stable sheaf on $\P^2$. Assume either $\rk(\xi)\leq 6$ or $\rk(\xi)$ and $c_1(\xi)$ are coprime and the discriminant $\Delta(\xi)$ is sufficiently large.  We use recent results of Bayer and Macr\`i \cite{BayerMacri2} on Bridgeland stability to compute the ample cone of the moduli space $M(\xi)$ of Gieseker semistable sheaves on $\P^2$.  We recover earlier results, such as those by Str\o mme \cite{Stromme} and Yoshioka \cite{Yoshioka}, as special cases.
\end{abstract}

\maketitle
\setcounter{tocdepth}{1}
\tableofcontents

\section{Introduction}

Let $\xi$ be the Chern character of a stable  sheaf on $\PP^2$. The moduli space $M(\xi)$ parameterizes $S$-equivalence classes of Gieseker semistable sheaves with Chern character $\xi$. It is an irreducible, normal, factorial, projective variety \cite{LePotierLectures}. In this paper, we determine the ample cone of $M(\xi)$  when either $\rk(\xi)\leq 6$ or $\rk(\xi)$ and $c_1(\xi)$ are coprime and the discriminant $\Delta(\xi)$ is sufficiently large.

The {\em ample cone} $\Amp(X)$ of a projective variety $X$ is the open convex cone in the N\'{e}ron-Severi space spanned by the classes of ample divisors. It controls embeddings of $X$ into projective space and is among the most important invariants of $X$. Its closure, the {\em nef cone} $\Nef(X)$, is spanned by divisor classes that have non-negative intersection with every curve on $X$ and is dual to the Mori cone of curves on $X$  under the intersection pairing (see \cite{Lazarsfeld}). We now describe our results on $\Amp(M(\xi))$ in greater detail.

Let $\xi$ be an integral Chern character with rank $r>0$.  We record such a character as a triple $(r,\mu,\Delta)$, where $$\mu = \frac{\ch_1}{r} \qquad \textrm{and}  \qquad \Delta = \frac{1}{2}\mu^2 - \frac{\ch_2}{r}$$ are the \emph{slope} and \emph{discriminant}, respectively.  We call the character $\xi$ (semi)stable if there exists a (semi)stable sheaf of character $\xi$. Dr\'{e}zet and Le Potier give an explicit curve $\delta(\mu)$ in the $(\mu, \Delta)$-plane such that the moduli space $M(\xi)$ is positive dimensional if and only if $\Delta \geq \delta(\mu)$ \cite{DLP}, \cite{LePotierLectures}.  The vector bundles whose Chern characters satisfy $\Delta = \delta(\mu)$ are called {\em height zero bundles}. Their moduli spaces have Picard group isomorphic to $\ZZ$. Hence, their ample cone is spanned by the positive generator and there is nothing further to discuss. Therefore, we will assume $\Delta>\delta(\mu)$, and say $\xi$ has \emph{positive height}.

There is a nondegenerate symmetric bilinear form on the $K$-group $K(\PP^2)$ sending a pair of Chern characters $\xi, \zeta$ to the Euler Characteristic $\chi(\xi^*, \zeta)$.  When $\xi$ has positive height, the Picard group of the moduli space $M(\xi)$  is naturally identified with the orthogonal complement $\xi^\perp$ and is isomorphic to $\ZZ \oplus \ZZ$ \cite{LePotierLectures}.  Correspondingly, the N\'{e}ron-Severi space is a two-dimensional vector space.  In order to describe $\Amp(M(\xi))$, it suffices to specify its two extremal rays.

The moduli space $M(\xi)$ admits a surjective, birational morphism $j: M(\xi)\rightarrow M^{DUY}(\xi)$ to the Donaldson-Uhlenbeck-Yau compactification $M^{DUY}(\xi)$ of the moduli space of stable vector bundles (see \cite{JunLiDonaldson} and \cite{HuybrechtsLehn}).  As long as the locus of singular (i.e., non-locally-free) sheaves  in $M(\xi)$ is nonempty (see Theorem \ref{thm-singular}), the morphism $j$ is not an isomorphism and contracts curves (see Proposition \ref{prop-DUY}).  Consequently, the line bundle $\mathcal{L}_1$ defining $j$  is base-point-free but not ample (see \cite{HuybrechtsLehn}).  It corresponds to a Chern character $u_1\in \xi^\perp \cong \Pic M(\xi)$ and spans an extremal ray of $\Amp(M(\xi))$.
For all the characters $\xi$ that we will consider in this paper there are singular sheaves in $M(\xi)$, so one edge of $\Amp(M(\xi))$ is always spanned by $u_1$.  We must compute the other edge of the cone, which we call the {\em primary edge}.

We now state our results.  Let $\xi = (r, \mu, \Delta)$ be a stable Chern character. Let  $\xi'= (r', \mu', \Delta')$ be the  stable Chern character satisfying the following defining properties:
\begin{itemize}
\item  $0< r' \leq r$ and  $\mu'< \mu$,
\item  Every rational number in the interval $(\mu', \mu)$ has denominator greater than $r$,
\item  The discriminant of any stable bundle of slope $\mu'$ and rank at most $r$ is at least $\Delta'$,
\item  The minimal rank of a stable Chern character with slope $\mu'$ and discriminant $\Delta'$ is $r'$.
\end{itemize} The character $\xi'$ is easily computed using Dr\'ezet and Le Potier's classification of stable bundles.  

\begin{theorem}\label{thm-asymptotic}
Let $\xi = (r,\mu,\Delta)$ be a positive height Chern character such that $r$ and $c_1$ are coprime.  Suppose $\Delta$ is sufficiently large, depending on $r$ and $\mu$.  The cone $\Amp(M(\xi))$ is spanned by $u_1$ and a negative rank character in $(\xi')^\perp$.
\end{theorem}

The required lower bound on $\Delta$ can be made explicit; see Remark \ref{rem-explicit}. Our second result computes the ample cone of small rank moduli spaces.

\begin{theorem}\label{thm-smallRank}
Let $\xi = (r,\mu,\Delta)$ be a positive height Chern character with $r\leq 6$.  
\begin{enumerate}
\item If $\xi$ is not a twist of $(6,\frac{1}{3},\frac{13}{18})$, then $\Amp(M(\xi))$ is spanned by $u_1$ and a negative rank character in $(\xi')^\perp$.
\item If $\xi = (6,\frac{1}{3},\frac{13}{18})$, then $\Amp(M(\xi))$ is spanned by $u_1$ and a negative rank character in $(\ch \OO_{\P^2})^\perp$.
\end{enumerate}
\end{theorem}

The new ingredient that allows us to calculate $\Amp(M(\xi))$ is Bridgeland stability. Bridgeland \cite{bridgeland:stable}, \cite{Bridgeland} and Arcara, Bertram \cite{ArcaraBertram} construct Bridgeland stability conditions on the bounded derived category of coherent sheaves on a projective surface.  On $\P^2$, these stability conditions $\sigma_{s,t} = (\cA_s, Z_{s,t})$ are parameterized by a half plane $H := \{ (s,t) | s,t \in \R, t>0\}$ (see \cite{ABCH} and \S \ref{sec-prelim}). Given a Chern character $\xi$, $H$ admits a finite wall and chamber decomposition, where in each chamber the collection of $\sigma_{s,t}$-semistable objects with Chern character $\xi$ remains constant.  These walls are disjoint and consist of a vertical line $s = \mu$ and nested semicircles with center along $t=0$ \cite{ABCH}.  In particular, there is a largest semicircular wall $W_{\max}$ to the left of the vertical wall.  We will call this wall the {\em Gieseker wall}.  Outside this wall, the moduli space $M_{\sigma_{s,t}}(\xi)$ of $\sigma_{s,t}$-semistable objects is isomorphic to the Gieseker moduli space $M(\xi)$ \cite{ABCH}.

Let $\sigma_0$ be a stability condition on the Gieseker wall for $M(\xi)$.  Bayer and Macr\`i \cite{BayerMacri2} construct a nef divisor $\ell_{\sigma_0}$ on $M(\xi)$ corresponding to $\sigma_0$.  They also characterize the curves $C$ in $M(\xi)$ which have intersection number $0$ with $\ell_{\sigma_0}$, as follows: $\ell_{\sigma_0}.C=0$ if and only if two general sheaves parameterized by $C$ are $S$-equivalent with respect to $\sigma_0$ (that is, their Jordan-H\"older factors with respect to the stability condition $\sigma_0$ are the same). The divisor $\ell_{\sigma_0}$ is therefore an extremal nef divisor if and only if such a curve in $M(\xi)$ exists.  This divisor can also be constructed via the GIT methods of Li and Zhao \cite{LiZhao}.
 
In light of the results of Bayer and Macr\`i, our proofs of Theorems \ref{thm-asymptotic} and \ref{thm-smallRank}  amount to the computation of the Gieseker wall.  For simplicity, we describe our approach to proving Theorem \ref{thm-asymptotic}; the basic strategy for the proof of Theorem \ref{thm-smallRank} is similar.

\begin{theorem}\label{thm-introWall}
Let $\xi$ be as in Theorem \ref{thm-asymptotic}.  The Gieseker wall for $M(\xi)$ is the wall $W(\xi',\xi)$ where $\xi$ and $\xi'$ have the same Bridgeland slope.
\end{theorem}

There are two parts to the proof of this theorem.  First, we show that $W_{\max}$ is no larger than $W(\xi',\xi)$.  This is a numerical computation based on Bridgeland stability.  The key technical result (Theorem \ref{thm-excludeHighRank}) is that if a wall is larger than $W(\xi',\xi)$, then the rank of a destabilizing subobject corresponding to the wall is at most $\rk(\xi)$.  We then find that the extremality properties defining $\xi'$ guarantee that $W(\xi',\xi)$ is at least as large as any wall for $M(\xi)$ (Theorem \ref{thm-main}). 

In the other direction, we must show that $W(\xi',\xi)$ is an actual wall for $M(\xi)$.  Define a character $\xi'' = \xi-\xi'$.  Our next theorem produces a sheaf $E\in M(\xi)$ which is destabilized along $W(\xi',\xi)$.

\begin{theorem}\label{thm-introSt}
Let $\xi$ be as in Theorem \ref{thm-asymptotic}.  Fix general sheaves  $F\in M(\xi')$ and $Q\in M(\xi'')$.  Then the general sheaf $E$ given by an extension $$0\to F\to E\to Q\to 0$$ is Gieseker stable. Furthermore, we obtain curves in $M(\xi)$ by varying the extension class.
\end{theorem}

If $E$ is a Gieseker stable extension as in the theorem, then $E$ is strictly semistable with respect to a stability condition $\sigma_0$ on $W(\xi',\xi)$, and not semistable with respect to a stability condition below $W(\xi',\xi)$.  Thus $W(\xi',\xi)$ is an actual wall for $M(\xi)$, and it is the Gieseker wall.  Any two Gieseker stable extensions of $Q$ by $F$ are $S$-equivalent with respect to $\sigma_0$, so any curve $C$ in $M(\xi)$ obtained by varying the extension class satisfies $\ell_{\sigma_0}.C=0$.  Therefore, $\ell_{\sigma_0}$ spans an edge of the ample cone. Dually, $C$ spans an edge of the Mori cone of curves.

The natural analogs of Theorems \ref{thm-introWall} and  \ref{thm-introSt} are almost true when instead $\rk(\xi)\leq 6$ as in Theorem \ref{thm-smallRank};  some minor adjustments to the statements need to be made for certain small discriminant cases.  See Theorems \ref{thm-smallRankCurves}, \ref{thm-mainsmall}, \ref{thm-mainSporadic}, and Propositions \ref{prop-sporadic} and \ref{prop-sporadic2} for precise statements.  As the rank increases beyond $6$, these exceptions become more common, and many more ad hoc arguments are required when using current techniques.

Bridgeland stability conditions were effectively used to study the birational geometry of Hilbert schemes of points on $\PP^2$ in \cite{ABCH} and moduli spaces of rank 0 semistable sheaves in \cite{Woolf}. The ample cone of $M(\xi)$ was computed earlier for some special Chern characters.  The ample cone of the Hilbert scheme of points on $\PP^2$ was computed in  \cite{li} (see also \cite{ABCH}, \cite{Ohkawa}).  Str\o mme computed $\Amp(M(\xi))$ when the rank of $\xi$ is two and either $c_1$ or $c_2 - \frac{1}{4}c_1^2$ is odd  \cite{Stromme}.    Similarly, when the slope is $\frac{1}{r}$, Yoshioka \cite{Yoshioka} computed the ample cone of $M(\xi)$ and described the first flip. Our results contain these as special cases.  Bridgeland stability has also been effectively used to compute ample cones of moduli spaces of sheaves on other surfaces. For example,  see \cite{ArcaraBertram}, \cite{BayerMacri2}, \cite{BayerMacri3}, \cite{MYY1}, \cite{MYY2}  for K3 surfaces,  \cite{MM}, \cite{Yoshioka2}, \cite{YanagidaYoshioka} for abelian surfaces, \cite{Nuer} for Enriques surfaces, and \cite{BertramCoskun} for the Hilbert scheme of points on Hirzebruch surfaces and del Pezzo surfaces.

\subsection*{Organization of the paper}
In \S \ref{sec-prelim}, we will introduce the necessary background on $M(\xi)$ and Bridgeland stability conditions on $\PP^2$.  In \S \ref{sec-hp} and \S \ref{sec-extremal}, we study the stability of extensions of sheaves and prove the first statement in Theorem \ref{thm-introSt}. In \S \ref{sec-elementaryMod} and \ref{sec-smallRank}, we prove the analogue of the first assertion in Theorem \ref{thm-introSt} for $\rk(\xi) \leq 6$. In \S \ref{sec-curves}, we complete the proof of Theorem \ref{thm-introSt} (and its small-rank analogue) by constructing the desired curves of extensions. Finally, in \S \ref{sec-ample}, we compute the Gieseker wall,  completing the proofs of Theorems \ref{thm-asymptotic} and \ref{thm-smallRank}.

\subsection*{Acknowledgements} We are grateful to Daniele Arcara, Arend Bayer, Aaron Bertram, Lawrence Ein, Joe Harris, Brendan Hassett, Emanuele Macr\`{i} and Matthew Woolf for many useful conversations about the geometry of moduli spaces of semistable sheaves. 

\section{Preliminaries}\label{sec-prelim}
In this section, we recall basic facts concerning the classification of stable vector bundles on $\PP^2$ and Bridgeland stability.

\subsection{Stable sheaves on $\PP^2$}
Let $\xi$ be the Chern character of a (semi)stable  sheaf on $\PP^2$. We will call such characters {\em (semi)stable characters}. The classification of stable characters on $\PP^2$ due to Dr\'{e}zet and Le Potier is best stated in terms of the slope $\mu$  and the discriminant $\Delta$. Let  $$P(m)= \frac{1}{2}(m^2 + 3m +2)$$ denote the Hilbert polynomial of $\OO_{\PP^2}$. In terms of these invariants, the Riemann-Roch formula reads $$\chi(E,F) = \rk(E) \rk(F) ( P( \mu(F) - \mu(E)) - \Delta(E) - \Delta(F)).$$

An {\em exceptional  bundle} $E$ on $\PP^2$ is a stable bundle such that $\Ext^1(E,E)=0$. The exceptional bundles are rigid; their moduli spaces consist of a single reduced point \cite[Corollary 16.1.5]{LePotierLectures}. They are the stable bundles $E$ on $\PP^2$ with $\Delta(E) < \frac{1}{2}$ \cite[Proposition 16.1.1]{LePotierLectures}. Examples of exceptional bundles include line bundles $\OO_{\PP^2}(n)$ and the tangent bundle $T_{\PP^2}$. All exceptional bundles can be obtained from line bundles via  a sequence of mutations \cite{DrezetBeilinson}.  An {\em exceptional slope} $\alpha\in \Q$ is the slope of an exceptional bundle. If $\alpha$ is an exceptional slope, there is a unique exceptional bundle $E_\alpha$ with slope $\alpha$. The rank of the exceptional bundle is the smallest positive integer $r_{\alpha}$ such that $r_{\alpha} \alpha$ is an integer. The discriminant $\Delta_{\alpha}$ is then given by $$\Delta_{\alpha} = \frac{1}{2} \left( 1 - \frac{1}{r_{\alpha}^2}\right).$$  The set $\F E$ of exceptional slopes is well-understood (see \cite{DLP} and \cite{CoskunHuizengaWoolf}).

The classification of positive dimensional moduli spaces of stable vector bundles on $\PP^2$ is expressed in terms of a fractal-like curve $\delta$ in the $(\mu, \Delta)$-plane. For each exceptional slope $\alpha \in \F E$, there is an interval $I_{\alpha} = ( \alpha-x_{\alpha}, \alpha + x_{\alpha})$ with $$x_{\alpha} = \frac{3- \sqrt{5+8 \Delta_{\alpha}}}{2}$$ such that the function $\delta(\mu)$ is defined on $I_{\alpha}$ by the function $$\delta(\mu) = P(-|\mu - \alpha|) - \Delta_{\alpha}, \ \ \mbox{if} \ \ \alpha - x_{\alpha} < \mu < \alpha + x_{\alpha}.$$ The graph of $\delta(\mu)$ is an increasing concave up parabola on the interval $[\alpha - x_{\alpha}, \alpha]$ and a decreasing concave up parabola on the interval $[\alpha, \alpha + x_{\alpha}]$.  The function $\delta$ is invariant under translation by integers. The main classification theorem of Dr\'{e}zet and Le Potier is as follows.

\begin{theorem}[\cite{DLP}, \cite{LePotierLectures}]
There exists a positive dimensional moduli space of Gieseker semistable sheaves $M(\xi)$ with integral Chern character $\xi$ if and only  if  $\Delta \geq \delta(\mu)$. In this case, $M(\xi)$ is a normal, irreducible, factorial projective variety of dimension $r^2(2 \Delta -1) + 1$.
\end{theorem}

\subsection{Singular sheaves on $\P^2$} For studying one extremal edge of the ample cone, we need to understand the locus of singular sheaves in $M(\xi)$. The following theorem, which is likely well-known to experts, characterizes the Chern characters where the locus of singular sheaves in $M(\xi)$ is nonempty. We include a proof for lack of a convenient reference.

\begin{theorem}\label{thm-singular}
Let $\xi= (r, \mu, \Delta)$ be an integral Chern character with $r>0$ and $\Delta \geq \delta(\mu)$. The locus of singular sheaves in $M(\xi)$ is empty if and only if $\Delta - \delta(\mu) < \frac{1}{r}$ and $\mu$ is not an exceptional slope.  
\end{theorem}

\begin{proof}
Let $F$ be a singular sheaf in $M(\xi)$. Then $F^{**}$ is a $\mu$-semistable, locally free sheaf \cite[\S 8]{HuybrechtsLehn} with invariants $$\rk(F^{**}) = r, \ \ \mu(F^{**}) = \mu, \ \ \mbox{and} \ \ \Delta(F^{**}) \leq \Delta(F) - \frac{1}{r}.$$

Since the set $\R - \cup_{\alpha \in \F{E}} I_{\alpha}$ does not contain any rational numbers \cite{DLP}, \cite[Theorem 4.1]{CoskunHuizengaWoolf},  $\mu \in I_{\alpha}$ for some exceptional slope $\alpha$. Let $E_{\alpha}$ with invariants $(r_{\alpha}, \alpha, \Delta_{\alpha})$ be the corresponding exceptional bundle.

If $\Delta - \delta(\mu)< \frac{1}{r}$ and $F$ is a singular sheaf in $M(\xi)$, then $\Delta(F^{**})< \delta(\mu)$. If $\alpha > \mu$, then $\hom(E_{\alpha}, F^{**})>0$. If $\alpha< \mu$, then $\hom(F^{**}, E_{\alpha})>0$. In either case, these homomorphisms violate the  $\mu$-semistability of $F^{**}$, leading to a contradiction. Therefore, if $\Delta - \delta(\mu) < \frac{1}{r}$ and $\mu$ is not an exceptional slope, then the locus of singular sheaves in $M(\xi)$ is empty.

To prove the converse, we construct singular sheaves using elementary modifications. If $\Delta - \delta(\mu) \geq \frac{1}{r}$, then $\zeta = (r, \mu, \Delta - \frac{1}{r})$ is a stable Chern character. Let $G$ be a $\mu$-stable bundle in $M(\zeta)$, which exists by \cite[Corollary 4.12]{DLP}. Choose a point $p\in \PP^2$ and let  $G \rightarrow \OO_p$ be a general surjection. Then the kernel sheaf $F$ defined by 
$$0 \rightarrow F \rightarrow  G \rightarrow \OO_p \rightarrow 0$$ is a $\mu$-stable, singular sheaf with Chern character $\xi$ (see \S \ref{sec-elementaryMod} for more details on elementary modifications).  

We are reduced to showing that if $\mu = \alpha$ and $\Delta - \delta(\alpha) < \frac{1}{r}$, then the locus of singular sheaves in $M(\xi)$ is nonempty. Since  $c_1(E_{\alpha})$ and  $r_{\alpha}$ are coprime, the rank of any bundle with slope $\alpha$ is a multiple of $r_{\alpha}$. Write $$r= k r_{\alpha}^2 + m r_{\alpha}, \ \ 0 \leq k, \ 0 < m \leq r_{\alpha}.$$ By integrality, there exists an integer $N$ such that $\Delta - \frac{N}{r} = \Delta_{\alpha}$. Our choice of $k$ implies $$ \Delta=  \Delta_{\alpha} + \frac{k+1}{r}.$$

First, assume $k=0$. If $r'<r$, then $\Delta_{\alpha} + \frac{1}{r'} > \Delta_{\alpha} + \frac{1}{r}$. Consequently, the only Gieseker semistable sheaves of character $(r',\mu,\Delta')$ with $r'<r$ and $\Delta'<\Delta$ are semi-exceptional sheaves $E_{\alpha}^{\oplus \ell}$ with $\ell < m$. Let $F$ be a general elementary modification of the form $$0 \rightarrow F \rightarrow E_{\alpha}^{\oplus m} \rightarrow \OO_p \rightarrow 0.$$ Then $F$ is a $\mu$-semistable singular sheaf with Chern character $\xi$. If $F$ were not Gieseker semistable, then it would admit an injective map $\phi: E_{\alpha} \rightarrow F$. By Lemma \ref{lem-Segre} below, for a general surjection $\psi: E_{\alpha}^{\oplus m} \rightarrow \OO_p$, there does not exist an injection $E_{\alpha} \rightarrow E_{\alpha}^{\oplus m}$ which maps to $0$ under $\psi$. Composing $\phi$ with the maps in the exact sequence defining $F$, we get a contradiction. We conclude that $F$ is Gieseker semistable. This constructs singular sheaves when $k=0$.

Next assume $k>0$. If $m= r_{\alpha}$, then we can construct a singular sheaf in $M(\xi)$ as a $(k+1)$-fold direct sum of a semistable singular sheaf constructed in the case $k=0$, $m=r_\alpha$.  Hence, we may assume that $m< r_{\alpha}$. Let $G$ be a $\mu$-stable vector bundle with Chern character $$\zeta= \left(kr_{\alpha}^2, \alpha, \delta(\alpha)=\Delta_{\alpha} +  \frac{1}{r_{\alpha}^2}\right).$$ Note that $(\mu(\zeta),\Delta(\zeta))$ lies on the curve $\delta$, hence $\chi(E_{\alpha}, G) = \chi(G, E_{\alpha}) =0$. Every locally free sheaf in $M(\zeta)$ has a two-step resolution in terms of exceptional bundles orthogonal to $E_{\alpha}$ \cite{DLP}. Consequently,  $\hom(G, E_{\alpha})=0$. We also have $\hom(E_\alpha, G)=0$ by stability. 

Let $ \phi: E_{\alpha}^{\oplus m} \oplus G \rightarrow \OO_p$ be a general surjection and let $F$ be defined as the corresponding elementary modification 
$$0 \rightarrow F \rightarrow  E_{\alpha}^{\oplus m} \oplus G \rightarrow \OO_p \rightarrow 0.$$  We first check that the Chern character of  $F$ is $\xi$. Clearly, $\rk(F)=r$ and $\mu(F)= \alpha$. The discriminant equals $$ \Delta(F) = \frac{1}{r} \left( m r_{\alpha} \Delta_{\alpha} +  k r_{\alpha}^2 \left(\Delta_{\alpha} + \frac{1}{r_{\alpha}^2}\right)\right) + \frac{1}{r} = \Delta_{\alpha} + \frac{k+1}{r} = \Delta.$$ Hence, $F$ is a singular sheaf with the correct invariants. It remains to check that it is Gieseker semistable. Note that $F$ is at least $\mu$-semistable.

Suppose $\psi: U \rightarrow F$ is an injection from a Gieseker stable sheaf $U$ that destabilizes $F$. Since $F$ is $\mu$-semistable, $\mu(U) = \alpha$ and $\Delta(U) < \Delta$. Then we claim that either $U = E_{\alpha}$ or $\rk(U) > m r_{\alpha}$. Suppose $U \not= E_{\alpha}$ and  $\rk(U) = s r_{\alpha}$. Then $$\Delta = \Delta_{\alpha} + \frac{k+1}{r}> \Delta(U) \geq \Delta_{\alpha} + \frac{1}{s r_{\alpha}}.$$ Hence, $$ s > \frac{k r_{\alpha} + m}{k+1} > \frac{km + m} {k+1} = m.$$ 

If $U \neq E_{\alpha}$, composing $\psi$ with the inclusion to $E_{\alpha}^{\oplus m} \oplus G$ gives an injection $\psi' : U \rightarrow E_{\alpha}^{\oplus m} \oplus G$. Since $\rk (U) > m r_{\alpha}$, the projection to $G$ cannot be zero. Hence, we get a nonzero map $\vartheta : U \rightarrow G$.  Let $V = \im \vartheta$. We have $\rk V = \rk G$ by the $\mu$-stability of $G$.  We claim that $\vartheta$ is in fact surjective.  The quotient $G/V$ is $0$-dimensional by stability, and, if it is nonzero, then $$\Delta(U) < \Delta < \delta(\alpha) + \frac{1}{r}< \delta(\alpha) + \frac{1}{kr_\alpha^2} \leq \Delta(V).$$ This violates the stability of $U$, so $V = G$ and $\vartheta$ is surjective.  If $\rk(U) = \rk(G)$, then $U\cong G$ and $\psi'$ maps $U$ isomorphically onto $G \subset E_{\alpha}^{\oplus m} \oplus G$.  A general hyperplane in the fiber $(E_\alpha^{\oplus m} \oplus G)_p$ is transverse to $G_p$, so this contradicts the fact that $\phi \circ \psi' = 0$ and $\phi$ is general.  Suppose $\rk(U)>\rk(G)$, and write $$\rk(U) = kr_{\alpha}^2+nr_\alpha$$ with $0< n<m$.  Then we find $$\Delta_\alpha+\frac{k+1}{r} =\Delta>\Delta(U) \geq \Delta_\alpha + \frac{k+1}{\rk(U)},$$ contradicting $\rk(U) < r$. We conclude that if $U\neq E_\alpha$, then $U$ cannot destabilize $F$.

On the other hand, if $U = E_{\alpha}$, then by the semistability of $G$ the composition of $\psi'$ with the projection to $G$ must be $0$. A general hyperplane in the fiber $(E_\alpha^{\oplus m} \oplus G)_p$ intersects $(E_{\alpha}^{\oplus m})_p$ in a hyperplane $H\subset (E_\alpha^{\oplus m})_p$.  Since $m\leq r_\alpha$, Lemma \ref{lem-Segre} shows the composition of $\psi'$ with $\phi$ is nonzero, a contradiction.  We conclude that $F$ is Gieseker semistable.  
\end{proof}

\begin{lemma}\label{lem-Segre}
Let $E_{\alpha}$ be an exceptional bundle of rank $r_{\alpha}$. Let $H$ be a general codimension $c$ subspace of the fiber of $E_{\alpha}^{\oplus m}$ over a point $p$. Then there exists an injection $\phi: E_{\alpha} \rightarrow E_\alpha^{\oplus m}$ such that $\phi_p(E_{\alpha}) \subset H$ if and only if $c r_{\alpha} \leq m-1$.
\end{lemma}
\begin{proof}
For simplicity set $E=E_\alpha $ and $r =r_\alpha$. Let $S$ denote the Segre embedding of $\PP^{r-1} \times \PP^{m-1}$ in $\PP^{rm-1}$. Let $q_1, q_2$ denote the two projections from $S$ to $\PP^{r-1}$ and $\PP^{m-1}$, respectively. We will call a linear $\PP^{r-1}$ in $S$ contracted by $q_2$ a {\em  $\PP^{r-1}$ fiber}. 

Let $\phi: E \rightarrow E^{\oplus m}$ be an injection. Composing $\phi$ with the $m$ projections, we get $m$ morphisms $E\rightarrow E$. Since $E$ is simple, the resulting maps are all homotheties. Let $M = (\lambda_1 I \ \ \lambda_2 I \ \ \dots \ \ \lambda_m I)$ be the $r \times rm$ matrix, where $I$ is the $r\times r$ identity matrix and $\lambda_i$ are scalars. Let $\vec{x} = (x_1, \dots, x_r)^T$. Hence, $\phi_p(E)$ has the form $$M \vec{x} = (\lambda_1 x_1, \lambda_1 x_2, \dots, \lambda_1 x_r, \dots, \lambda_m x_1, \dots, \lambda_m x_r)^T.$$ If we projectivize, we see that the fibers $\PP (\phi_p(E))$ are  $\PP^{r -1}$ fibers contained in the Segre embedding of $\PP^{r -1} \times \PP^{m-1}$ in $\PP((E^{\oplus m})_p)$. Conversely, every  $\PP^{r-1}$ fiber in $S$ is obtained by fixing a point $(\lambda_1, \dots, \lambda_m) \in \PP^{m-1}$  and, hence, is the fiber of an injection $E \rightarrow E^{\oplus m}$.  The lemma thus reduces to the statement that a general codimension $c$ linear subspace of $\PP^{m r -1}$ contains a   $\PP^{r -1}$ fiber in $S$ if and only if $c r \leq m-1$.

Consider the incidence correspondence $$J = \{ (A, H) : H \cong \PP^{m r - 1 -c}, A \subset H \cap S \ \  \mbox{is a } \ \PP^{r-1} \  \mbox{fiber}\}.$$ Then the first projection $\pi_1$ maps $J$ onto $\PP^{m-1}$. The fiber of $\pi_1$ over a linear space $A$ is the set of codimension $c$ linear spaces that contain $A$, hence it is  isomorphic to the Grassmannian $G((m-1)r - c, (m-1)r)$. By the theorem on the dimension of fibers, $J$ is irreducible of dimension $(cr+1)(m-1) - c^2$. The second projection cannot dominate $G(mr-c, mr)$ if $\dim(J) < \dim(G(mr-c, mr)= c(mr-c)$. Comparing the two inequalities, we conclude that if $c r > m-1$, the second projection is not dominant. Hence, the general codimension $c$ linear space does not contain  a  $\PP^{r-1}$ fiber in $S$.  

To see the converse, we check that if  $r \leq m-1$, then a general hyperplane contains a codimension $r$ locus of linear $\PP^{r-1}$ fibers of $S$.  Consider the hyperplane $ H$ defined by $\sum_{i=1}^{r} Z_{(i-1)r + i} =0.$ Substituting the equations of the  Segre embedding, we see that $\sum_{i=1}^{r} \alpha_i x_i = 0$. Since this equation must hold for every choice of $x_i$, we conclude that $\alpha_1 = \cdots = \alpha_r =0$. Hence, the locus of  $\PP^{r-1}$ fibers of $S$ contained in $H$ is a codimension $r$ linear space in $\PP^{m-1}$.  A codimension $c$ linear space is the intersection of $c$ hyperplanes. Moreover,  the intersection of $c$ codimension $r$ subvarieties of $\PP^{m-1}$ is nonempty if $c r \leq m-1$. Hence, if $c r \leq m-1$, every codimension $c$ linear space contains a   $\PP^{r-1}$ fiber of $S$. This suffices to prove the converse. 
 \end{proof}

\subsection{The Picard group and Donaldson-Uhlenbeck-Yau compactification} Stable vector bundles with $\Delta = \delta(\mu)$  are called {\em height zero} bundles. Their moduli spaces have Picard rank one. The ample generator spans the ample cone and there is nothing further to discuss. 

For the rest of the subsection, suppose $\xi=(r,\mu,\Delta)$ is a \emph{positive height} Chern character, meaning $\Delta> \delta(\mu)$.  There is a  pairing on $K(\PP^2)$ given by $(\xi, \zeta) = \chi(\xi^*, \zeta)$. When $\Delta > \delta(\mu)$, Dr\'{e}zet proves that the Picard group of $M(\xi)$ is a free abelian group on two generators naturally identified with $\xi^{\perp}$ in $K(\PP^2)$ \cite{LePotierLectures}.  In $M(\xi)$, linear equivalence and numerical equivalence coincide and the N\'{e}ron-Severi space $\NS(M(\xi)) = \Pic(M(\xi)) \otimes \R$ is a two-dimensional vector space. In order to specify the ample cone, it suffices to determine its two extremal rays. 

In $\xi^\perp \cong \Pic(M(\xi))$ there is a unique character $u_1$ with $\rk(u_1) = 0$ and $c_1(u_1) = -r$.  The corresponding line bundle $\mathcal L_1$ is base-point-free and defines the Jun Li morphism $j: M(\xi) \to M^{DUY}(\xi)$ to the Donaldson-Uhlenbeck-Yau compactification \cite[\S 8]{HuybrechtsLehn}.

\begin{proposition}\label{prop-DUY}
Let $\xi= (r,\mu,\Delta)$ be a positive height character, and suppose that there are singular sheaves in $M(\xi)$. Then $u_1$ spans an extremal edge of $\Amp(M(\xi))$.
\end{proposition}

\begin{proof}
We show that $j$ contracts a curve in $M(\xi)$.  Two stable sheaves $E,E'\in M(\xi)$ are identified by $j$ if $E^{**} \cong (E')^{**}$ and the sets of singularities of $E$ and $E'$ are the same (counting multiplicity). The proof of Theorem \ref{thm-singular} constructs singular sheaves via an elementary modification that arises from a surjection $E= E_{\alpha}^{\oplus m} \oplus G \rightarrow \OO_p$. Here $m=0$ if $\Delta - \delta(\mu) \geq \frac{1}{r}$ or $\mu$ is not exceptional. Otherwise, $1 \leq m < r_{\alpha}$. Note that $\hom(E, \OO_p)= r$ and $\dim (\Aut(E)) = m^2 +1$ if $G\not= 0$ and $\dim (\Aut(E)) = m^2$ if $G=0$.  Hence, if $r>1$,  varying the surjection $E \rightarrow \OO_p$ gives a positive dimensional family of nonisomorphic  Gieseker semistable sheaves with the same singular support and double dual.  If instead $r = 1$ and $\Delta\geq 2$, then (up to a twist) $j$ is the Hilbert-Chow morphism to the symmetric product, and the result is still true.  
\end{proof}

\begin{corollary}\label{cor-DUY}
Let $\xi = (r,\mu,\Delta)$ be a positive height character.  If $\Delta$ is sufficiently large or if $r\leq 6$ then $u_1$ spans an edge of $\Amp(M(\xi))$.
\end{corollary}
\begin{proof}
In either case, this follows from Theorem \ref{thm-singular} and Proposition \ref{prop-DUY}.
\end{proof}

\subsection{Bridgeland stability conditions on $\P^2$}

We now recall basic facts concerning Bridgeland stability conditions on $\PP^2$ developed in \cite{ABCH}, \cite{CoskunHuizenga} and \cite{HuizengaPaper2}.

A {\em Bridgeland stability condition} $\sigma$ on the bounded derived category $\cD^b(X)$ of coherent sheaves on a smooth projective variety $X$ is a pair $\sigma = (\cA, Z)$, where $\cA$ is the heart of a bounded $t$-structure and $Z$ is a group homomorphism $$Z: K(\cD^b(\PP^2)) \rightarrow \C$$ satisfying the following two properties.

\begin{enumerate}
\item (Positivity) For every object $0 \not= E \in \cA$, $Z(E) \in \{r e^{i \pi \theta} | r> 0, 0 < \theta \leq 1\}$. 
Positivity allows one to define the slope of a non-zero object in $\cA$ by setting $$\mu_Z(E) =  - \frac{\Re(Z(E))}{\Im(Z(E))}.$$ An object $E$  of $\cA$ is called {\em (semi)stable} if  for every proper subobject $F\subset E$ in $\cA$ we have $\mu_Z(F) < (\leq) \mu_Z (E)$. 

\item (Harder-Narasimhan Property) Every object of $\cA$ has a finite Harder-Narasimhan filtration. 
\end{enumerate}

Bridgeland \cite{Bridgeland} and Arcara and Bertram \cite{ArcaraBertram} have constructed Bridgeland stability conditions on projective surfaces. In the case of $\PP^2$, the relevant Bridgeland stability conditions have the following form. Any torsion-free coherent sheaf $E$ on $\PP^2$ has a Harder-Narasimhan filtration $$0= E_0 \subset E_1 \subset \cdots \subset E_n = E$$ with respect to the Mumford slope with semistable factors $\gr_i = E_i / E_{i-1}$ such that $$\mu_{\max}(E) = \mu(\gr_1) > \cdots > \mu(\gr_n) = \mu_{\min}(E) .$$ Given $s\in \R$, let $\cQ_s$ be the full subcategory of $\coh(\PP^2)$ consisting of sheaves such that their quotient by their torsion subsheaf have $\mu_{\min}(Q)> s$.  Similarly, let $\cF_s$  be the full subcategory of $\coh(\PP^2)$ consisting of  torsion free sheaves $F$ with $\mu_{\max}(F) \leq s$.  Then the abelian category  $$\cA_s := \{ E \in \cD^b(\PP^2) : \rH^{-1}(E) \in \cF_s, \rH^0(E) \in \cQ_s, H^i(E) = 0 \ \mbox{for} \ i \not= -1, 0 \}$$ obtained by tilting the category of coherent sheaves with respect to the torsion pair $(\cF_s, \cQ_s)$ is the heart of a bounded $t$-structure. Let $$Z_{s,t} (E) = - \int_{\PP^2} e^{-(s+it)H} \ch(E),$$ where $H$ is the hyperplane class on $\PP^2$. The pair $(\cA_s, Z_{s,t})$ is a Bridgeland stability condition for every $s > 0$ and $t \in \R$. We thus obtain a half plane of Bridgeland stability conditions on $\PP^2$ parameterized by $(s, t)$, $t>0$.

\subsection{Bridgeland walls}
If we fix a Chern character $\xi\in K(\P^2)$, the $(s,t)$-plane of stability conditions for $\P^2$ admits a finite wall and chamber structure where the objects in $\cA_s$ with Chern character $\xi$ that are stable with respect to the stability condition $(\cA_s, Z_{s,t})$ remain unchanged within the interior of a chamber (\cite{ABCH}, \cite{Bridgeland}, \cite{BayerMacri}, \cite{BayerMacri2}).  An object $E$ is destabilized along a wall  $W(E, F)$ by $F$  if $E$ is semistable on one side of the wall but $F \subset E$ in the category $\cA_s$ with $\mu_{s,t} (F) > \mu_{s,t}(E)$ on the other side of the wall.  We call these walls {\em Bridgeland walls}. The equations of the wall $W(E,F)$ can be computed using the relation  $\mu_{s,t} (F) = \mu_{s,t}(E)$ along the wall. 

Suppose $\xi,\zeta\in K(\P^2)\te \R$ are two linearly independent real Chern characters.  A \emph{potential Bridgeland wall} is a set in the $(s,t)$-half-plane of the form $$W(\xi,\zeta) = \{(s,t):\mu_{s,t}(\xi) = \mu_{s,t}(\zeta)\},$$ where $\mu_{s,t}$ is the slope associated to $\cZ_{s,t}$.  Bridgeland walls are always potential Bridgeland walls.  The \emph{potential Bridgeland walls for $\xi$} are all the potential walls $W(\xi,\zeta)$ as $\zeta$ varies in $K(\P^2)\te \R$.  If $E,F\in D^b(\P^2)$, we also write $W(E,F)$ as a shorthand for $W(\ch(E),\ch(F))$.

The potential walls $W(\xi,\zeta)$ can be easily computed in terms of the Chern characters $\xi$ and $\zeta $.  
\begin{enumerate} \item If $\mu(\xi) = \mu(\zeta)$ (where the Mumford slope is interpreted as $\infty$ if the rank is $0$), then the wall $W(\xi,\zeta)$ is the vertical line $s= \mu(\xi)$ (interpreted as the empty set when the slope is infinite).  
\item Otherwise, without loss of generality assume $\mu(\xi)$ is finite, so that $r\neq 0$.  The walls $W(\xi,\zeta)$ and $W(\xi,\xi+\zeta)$ are equal, so we may further reduce to the case where both $\xi$ and $\zeta$ have nonzero rank. Then we may encode $\xi = (r_1,\mu_1,\Delta_1)$ and $\zeta = (r_2,\mu_2,\Delta_2)$ in terms of slope and discriminant instead of $\ch_1$ and $\ch_2$.  The wall $W(\xi,\zeta)$ is the semicircle centered at the point $(s,0)$ with $$s = \frac{1}{2}(\mu_1+\mu_2)-\frac{\Delta_1-\Delta_2}{\mu_1-\mu_2}$$ and having radius $\rho$ given by $$\rho^2 = (s-\mu_1)^2-2\Delta_1.$$ 
\end{enumerate}

In the principal case of interest, the Chern character $\xi = (r,\mu,\Delta)$ has nonzero rank $r\neq 0$ and nonnegative discriminant $\Delta\geq 0$.  In this case, the potential walls for $\xi$ consist of a vertical wall $s=\mu$ together with two disjoint nested families of semicircles on either side of this line \cite{ABCH}.  Specifically, for any $s$ with $|s-\mu| > \sqrt{2\Delta}$, there is a unique semicircular potential wall with center $(s,0)$ and radius $\rho$ satisfying $$\rho^2 = (s-\mu)^2 - 2\Delta.$$ The semicircles are centered along the $s$-axis, with smaller semicircles having centers closer to the vertical wall.  Every point in the $(s,t)$-half-plane lies on a unique potential wall for $\xi$.  When $r>0$, only the family of semicircles left of the vertical wall is interesting, since an object $E$ with Chern character $\xi$ can only be in categories $\cA_s$ with $s<\mu$.

Since  the number of Bridgeland walls is finite, there exists a largest semicircular Bridgeland wall $W_{\max}$ to the left of the vertical line $s = \mu$ that contains all other semicircular walls. Furthermore, for every $(s,t)$ with $s< \mu$ and contained outside $W_{\max}$, the moduli space of Bridgeland stable objects in $\cA_s$ with respect to $Z_{s,t}$ and Chern character $\xi$ is isomorphic to the moduli space $M(\xi)$ \cite{ABCH}.  We call $W_{\max}$ the \emph{Gieseker wall}.

\subsection{A nef divisor on $M(\xi)$}\label{ssec-BayerMacriPlan} Let $(\cA,Z) = \sigma_0  \in W_{\max}$ be a stability condition on the Gieseker wall.  Bayer and Macr\`{i} \cite{BayerMacri2} construct a nef divisor $\ell_{\sigma_0}$ on $M(\xi)$ corresponding to $\sigma_0$.  They also compute its class and describe geometrically the curves $C\subset M(\xi)$ with $C \cdot \ell_{\sigma_0} = 0$.

To describe the class of $\ell_{\sigma_0}$ in $\xi^\perp \cong \Pic M(\xi)$, consider the functional \begin{align*}N^1(M(\xi)) & \to \R\\
\xi' &\mapsto \Im\left( -\frac{Z(\xi')}{Z(\xi)}\right).\end{align*} Since the pairing $(\xi,\zeta) = \chi(\xi\te \zeta)$ is nondegenerate, we can write this functional as $(\zeta,-)$ for some unique $\zeta\in \xi^\perp$.  In terms of the isomorphism $\xi^\perp \cong \Pic M(\xi)$, we have $\zeta = [\ell_{\sigma_0}].$  Considering $(\zeta,\ch \OO_p)$ shows that $\zeta$ has negative rank.  Furthermore, if $W_{\max} = W(\xi',\xi)$ (so that $Z(\xi')$ and $Z(\xi)$ are real multiples of one another), then $\zeta$ is a negative rank character in $(\xi')^\perp$.  The ray in $N^1(M(\xi))$ determined by $\sigma_0$ depends only on $W_{\max}$, and not the particular choice of $\sigma_0$.

A curve $C\subset M(\xi)$ is orthogonal to $\ell_{\sigma_0}$ if and only if two general sheaves parameterized by $C$ are $S$-equivalent with respect to $\sigma_0$.  This gives an effective criterion for determining when the Bayer-Macr\`i divisor $\ell_{\sigma_0}$ is an extremal nef divisor.  In every case where we compute the ample cone of $M(\xi)$, the divisor $\ell_{\sigma_0}$ is in fact extremal.

\section{Admissible decompositions}\label{sec-hp}

In this section, we introduce the notion of an admissible decomposition of a Chern character of positive rank. Each such decomposition corresponds to a potential Bridgeland wall. In the cases when we can compute the ample cone, the Gieseker wall will correspond to a certain admissible decomposition.

\begin{definition}\label{def-admissible} Let $\xi$ be a stable Chern character of positive rank.  A \emph{decomposition} of $\xi$ is a triple $\Xi = (\xi',\xi,\xi'')$ such that $\xi = \xi'+\xi''$.  We say $\Xi$ is an \emph{admissible decomposition} if furthermore \begin{enumerate}[label=(D\arabic*)]
\item \label{cond-Fstable} $\xi'$ is semistable,
\item \label{cond-Qstable} $\xi''$ is stable,
\item\label{cond-rank} $0 < \rk(\xi') \leq \rk(\xi)$, 
\item \label{cond-Fslope} $\mu(\xi') < \mu(\xi)$,  and
\item \label{cond-slopeDiff} if $\rk(\xi'')>0$, then $\mu(\xi'')-\mu(\xi')<3$.
\end{enumerate}
\end{definition}

\begin{remark} The Chern characters in an admissible decomposition $\Xi$ span a $2$-plane in $K(\P^2)$.  We write $W(\Xi)$ for the potential Bridgeland wall where characters in this plane have the same slope.

Condition \ref{cond-Fstable} means that $\xi'$ is either semiexceptional or stable.  We require $\xi''$ to be stable since this holds in all our examples and makes admissibility work better with respect to elementary modifications; see \S\ref{sec-elementaryMod}.
\end{remark}

There are a couple numerical properties of decompositions which will frequently arise.

\begin{definition}\label{def-numericalProps}
Let $\Xi = (\xi',\xi,\xi'')$ be a decomposition.  
\begin{enumerate}
\item $\Xi$ is \emph{coprime} if $\rk(\xi)$ and $c_1(\xi)$ are coprime.

\item $\Xi$ is \emph{torsion} if $\rk(\xi'')=0$, and \emph{torsion-free} otherwise.
\end{enumerate}
\end{definition}

The conditions in the definition of an admissible decomposition ensure that there is a well-behaved space of extensions of the form $$0\to F\to E\to Q\to 0$$ with $F\in M(\xi')$ and $Q\in M(\xi'')$.

\begin{lemma}\label{existenceOfExtensions}
Let $\Xi = (\xi',\xi,\xi'')$ be an admissible torsion-free decomposition.  We have $\chi(\xi'',\xi')<0$.  In particular, for any $F\in M(\xi')$ and $Q\in M(\xi'')$ there are non-split extensions $$0\to F\to E \to Q \to 0.$$ Furthermore, $\Ext^1(Q,F)$ has the expected dimension $-\chi(\xi'',\xi')$ for any $F\in M(\xi')$ and $Q\in M(\xi'')$.
\end{lemma}
\begin{proof}
From \ref{cond-Fslope} and the torsion-free hypothesis, we have $\mu(\xi)<\mu(\xi'')$.  Let $F\in M(\xi')$ and $Q\in M(\xi'')$.  By stability, $\Hom(Q,F) = 0$.  Using Serre duality with condition \ref{cond-slopeDiff}, we have $\Ext^2(Q,F)=0$.  Therefore $\ext^1(Q,F) = -\chi(\xi'',\xi')$ and $\chi(\xi'',\xi')\leq 0$.

To prove $\chi(\xi'',\xi')<0$, first suppose $\xi'$ is semiexceptional.  Then $$\chi(\xi'',\xi')=\chi(\xi,\xi')-\chi(\xi',\xi')<\chi(\xi,\xi').$$ As in the previous paragraph, $\chi(\xi,\xi')\leq 0$, hence $\chi(\xi'',\xi')<0$.
 A similar argument works if $\xi''$ is semiexceptional.

Assume neither $\xi'$ or $\xi''$ is semiexceptional.  Then $-3<\mu(\xi')-\mu(\xi'')<0$  and $\Delta(\xi')+\Delta(\xi'')>1$.  Since $P(x)<1$ for $-3< x < 0 $, we conclude $\chi(\xi'',\xi')<0$ by the Riemann-Roch formula.
\end{proof}

We now introduce a notion of stability for an admissible decomposition $\Xi$.  Let $F_{s'}/S'$ (resp. $Q_{s''}/S''$) be a complete flat family of semistable sheaves with Chern character $\xi'$ (resp. $\xi''$), parameterized by a smooth and irreducible base variety.  Since $\ext^1(Q_{s''},F_{s'})$ does not depend on $(s',s'')\in S'\times S''$, there is a projective bundle $S$ over $S'\times S''$ such that the fiber over a point $(s',s'')$ is $\P\Ext^1(Q_{s''},F_{s'})$.  Then $S$ is smooth, irreducible, and it carries a universal extension sheaf $E_s/S$.

We wish to examine the stability properties of the general extension $E_s/S$.  If $E_s$ is \mbox{$(\mu$-)(semi)stable} for some $s\in S$, then the general $E_s$ has the same stability property.  Since the moduli spaces $M(\xi')$ and $M(\xi'')$ are irreducible, the general $E_s$ will be $(\mu$-)(semi)stable if and only if there exists some extension $$0\to F \to E\to Q\to 0$$ where $F\in M(\xi')$, $Q\in M(\xi'')$, and $E$ is $(\mu$-)(semi)stable. Since $S$ is complete,  we do not need to know that $E$ is parameterized by a point of $S$.

\begin{definition}\label{def-stableTriple}
Let $\Xi$ be an admissible decomposition.  We say that $\Xi$ is \emph{generically} \mbox{$(\mu$-)}(semi)stable if there is some extension $$0\to F\to E\to Q\to 0$$ where $F\in M(\xi')$, $Q\in M(\xi'')$, and $E$ is $(\mu$-)(semi)stable.
\end{definition}

\section{Extremal triples}\label{sec-extremal}  We now introduce the decomposition of a Chern character $\xi$ which frequently corresponds to the primary edge of the ample cone of $M(\xi)$.

\begin{definition}\label{def-extremal}
We call a triple $\Xi=(\xi',\xi,\xi'')$ of Chern characters \emph{extremal} if it is an admissible decomposition of $\xi$ with the following additional properties:
\begin{enumerate}[label=(E\arabic*)]
\item \label{cond-slopeClose} $\xi'$ and $\xi$ are \emph{slope-close}: we have $\mu(\xi') < \mu(\xi)$, and every rational number in the interval $(\mu(\xi'),\mu(\xi))$ has denominator larger than $\rk(\xi)$.
\item \label{cond-discMinimal} $\xi'$ is \emph{discriminant-minimal}: if $\theta'$ is a stable Chern character with $0<\rk(\theta')\leq \rk(\xi)$ and $\mu(\theta') = \mu(\xi')$, then $\Delta(\theta')\geq \Delta(\xi')$.
\item \label{cond-rankMinimal} $\xi'$ is \emph{rank-minimal}: if $\theta'$ is a stable Chern character with $\mu(\theta')=\mu(\xi')$ and $\Delta(\theta') = \Delta(\xi')$, then $\rk(\theta')\geq \rk(\xi')$.
\end{enumerate}
\end{definition}

\begin{remark}\label{rem-extremalRemark}
If $\Xi$ is an extremal triple, then it is uniquely determined by $\xi$.   The wall $W(\Xi)$ thus also only depends on $\xi$.  Not every stable character $\xi$ can be decomposed into an extremal triple $\Xi = (\xi',\xi,\xi'')$, but the vast majority can; see Lemma \ref{lem-extremalExist}.

Condition \ref{cond-discMinimal} in Definition \ref{def-admissible} is motivated by the formula for the center $(s,0)$ of $W(\Xi)$: $$s = \frac{\mu(\xi')+\mu(\xi)}{2}-\frac{\Delta(\xi')-\Delta(\xi)}{\mu(\xi')-\mu(\xi)}.$$ If $\Delta(\xi')$ decreases while the other invariants are held fixed, then the center of $W(\Xi)$ moves left.  Correspondingly, the wall becomes larger.  As we are searching for the largest walls, intuitively we should restrict our attention to triples with minimal $\Delta(\xi')$.  

Similarly, condition \ref{cond-slopeClose} typically helps make the wall $W(\Xi)$ large.  In the formula for $s$, the term $$-\frac{\Delta(\xi')-\Delta(\xi)}{\mu(\xi')-\mu(\xi)}$$ will dominate the expression if $\Delta(\xi)$ is sufficiently large and $\mu(\xi')$ is sufficiently close to $\mu(\xi)$.

Condition \ref{cond-rankMinimal} forces $\xi'$ to be stable, since semiexceptional characters are multiples of exceptional characters.
\end{remark}

The next lemma shows the definition of an extremal triple is not vacuous. 
\begin{lemma}\label{lem-extremalExist}
Let $\xi = (r,\mu,\Delta)$ be a stable Chern character, and suppose either
\begin{enumerate}
\item $\Delta$ is sufficiently large (depending on $r$ and $\mu$) or
\item $r\leq 6$.
\end{enumerate}
Then there is a unique extremal triple $\Xi = (\xi',\xi,\xi'')$.
\end{lemma}
\begin{proof}
Let $(r^\bullet, \mu^\bullet, \Delta^\bullet)$ denote the rank, slope and discriminant of $\xi^\bullet$.  The Chern character $\xi'$ is uniquely determined by conditions \ref{cond-Fstable}, \ref{cond-rank}, and \ref{cond-slopeClose}-\ref{cond-rankMinimal};  it depends only on $r$ and $\mu$, and not $\Delta$.  Set $\xi''=\xi-\xi'$, and observe that $r''$ and $\mu''$ depend only on $r$ and $\mu$.  We must check that $\xi''$ is stable and $\mu''-\mu'<3$ if $r''>0$. If $r''=0$, then $c_1(\xi'')>0$, so stability is automatic.

  Suppose $r''>0$.  Let us show $\mu''-\mu'<3$.  By \ref{cond-slopeClose} we have $\mu'\geq\mu-\frac{1}{r}$, so $$r''\mu''=r\mu - r'\mu'\leq (r-r')\mu+\frac{r'}{r}<r''\mu+1$$ and $$\mu''-\mu' < \mu+\frac{1}{r''}-\mu+\frac{1}{r} = \frac{1}{r''}+\frac{1}{r}\leq \frac{3}{2}.$$

If $r\leq 6$, we will see that $\xi''$ is stable in \S\ref{sec-smallRank}.  Suppose $\Delta$ is sufficiently large.  We have a relation $$r\Delta = r'\Delta'+r''\Delta''-\frac{r'r''}{r}(\mu'-\mu'')^2.$$  The invariants  $r',\mu',\Delta',r'',\mu''$ depend only on $r$ and $\mu$.   By making $\Delta$ large, we can make $\Delta''$ as large as we want, and thus we can make $\xi''$ stable.
\end{proof}

It is easy to prove a weak stability result for extremal triples.

\begin{proposition}\label{prop-slopeSemistable}
Let $\Xi=(\xi',\xi,\xi'')$ be an extremal torsion-free triple.  Then $\Xi$ is generically $\mu$-semistable.
\end{proposition}
\begin{proof}
By Lemma \ref{existenceOfExtensions}, there is a non-split extension $$0\to F \to E \to Q\to 0$$ with $F\in M^s(\xi')$ stable and $Q\in M(\xi)$.  We will show $E$ is $\mu$-semistable.  Since $F$ and $Q$ are torsion-free, $E$ is torsion-free as well.

Suppose $E$ is not $\mu$-semistable.  Then there is some surjection $E\to C$ with $\mu(C)<\mu(E)$ and $\rk(C)<\rk(E)$.  By passing to a suitable quotient of $C$, we may assume $C$ is stable.  Using slope-closeness \ref{cond-slopeClose}, we find $\mu(C)\leq \mu(F)$.  

First assume $\mu(C)<\mu(F)$.  By stability, the composition $F\to E\to C$ is zero, and thus $E\to C$ induces a map $Q\to C$.  This map is zero by stability, from which we conclude $E\to C$ is zero, a contradiction.

Next assume $\mu(C)= \mu(F)$.  If $\Delta(C)>\Delta(F)$, then we have an inequality $p_C<p_F$ of reduced Hilbert polynomials, so $F\to C$ is zero by stability and we conclude as in the previous paragraph.  On the other hand, $\Delta(C)<\Delta(F)$ cannot occur by the minimality condition \ref{cond-discMinimal}.

Finally, suppose $\mu(C) = \mu(F)$ and $\Delta(C) = \Delta(F)$.  Since $C$ and $F$ are both stable, any nonzero map $F\to C$ is an isomorphism.  Then the composition $E\to C\to F$ with the inverse isomorphism splits the sequence.
\end{proof}

The following corollary gives the first statement of Theorem \ref{thm-introSt} in the torsion-free case.

\begin{corollary}\label{cor-slopeStable}
If $\Xi$ is a coprime, torsion-free, extremal triple, then it is generically $\mu$-stable.
\end{corollary}

\section{Elementary modifications}\label{sec-elementaryMod} Many stability properties of an admissible decomposition $\Xi=(\xi',\xi,\xi'')$ are easier to understand when the discriminant $\Delta(\xi)$ is small.  Elementary modifications allow us to reduce to the small discriminant case.

\begin{definition}
Let $G$ be a coherent sheaf and let $G\to \OO_p$ be a surjective homomorphism.  Then the kernel $$0\to G'\to G\to \OO_p\to 0$$ is called an \emph{elementary modification} of $G$.    
\end{definition}

If $G$ has positive rank, we observe the equalities $$\rk(G') = \rk(G) \qquad \mu(G') = \mu(G) \qquad \Delta(G') = \Delta(G) + \frac{1}{\rk(G)} \qquad \chi(G') =\chi(G)-1.$$ The next lemma is immediate.

\begin{lemma}
If $G$ is $\mu$-(semi)stable, then any elementary modification of $G$ is $\mu$-(semi)stable.
\end{lemma}

\begin{warning}
Elementary modifications do not generally preserve Gieseker (semi)stability.  This is our reason for focusing on $\mu$-stability of extensions.
\end{warning}

Given a short exact sequence of sheaves, there is a natural induced sequence involving compatible elementary modifications.

\begin{proposition-definition}\label{def-elementaryModSequence}
Suppose $$0\to F \to E \to Q\to 0$$ is a short exact sequence of sheaves.  Let $Q'$ be the elementary modification of $Q$ corresponding to a homomorphism $Q\to \OO_p$, and let $E'$ be the elementary modification of $E$ corresponding to the composition $E\to Q\to \OO_p$.  Then there is a natural short exact sequence $$0\to F \to E'\to Q'\to 0.$$ This sequence is called an \emph{elementary modification} of the original sequence.
\end{proposition-definition}
\begin{proof}
A straightforward argument shows that there is a natural commuting diagram
$$\xymatrix{
&&0\ar[d]&0\ar[d]&\\
0\ar[r]&F\ar@{=}[d]\ar[r]&E'\ar[d]\ar[r]&Q'\ar[d]\ar[r]&0\\
0\ar[r]&F\ar[r]&E\ar[d]\ar[r]&Q\ar[d]\ar[r]&0\\
&&\OO_p\ar[d]\ar@{=}[r]&\OO_p\ar[d]&\\
&&0&0&
}$$
with exact rows and columns.
\end{proof}

We similarly extend the notion of elementary modifications to decompositions of Chern characters.

\begin{definition}
Let $\Xi = (\xi',\xi,\xi'')$ be a decomposition.  Let $\Theta = (\theta',\theta,\theta'')$ be the decomposition  such that
\begin{enumerate}
\item $\theta'=\xi'$,
\item $\theta$ and $\xi$ have the same rank and slope, and
\item $\Delta(\theta) = \Delta(\xi) + \frac{1}{\rk(\xi)}$.
\end{enumerate}
We call $\Theta$ the \emph{elementary modification} of $\Xi$.  If $\Xi$ is admissible, then $\Theta$ is admissible as well.

If $\Xi$ and $\Theta$ are admissible decompositions, we say $\Theta$ lies \emph{above} $\Xi$, and write $\Xi\preceq \Theta$, if conditions (1)-(3) are satisfied and $\Delta(\xi)\leq \Delta(\theta)$. Finally, $\Xi$ is \emph{minimal} if it is a minimal admissible decomposition with respect to $\preceq$.
\end{definition}

The next result follows from the integrality of the Euler characteristic and the Riemann-Roch formula.

\begin{lemma}
Let $\Xi$ and $\Theta$ be admissible decompositions.  Then $\Xi\preceq \Theta$ if and only if $\Theta$ is an iterated elementary modification of $\Xi$.
\end{lemma}

Extremality is preserved by elementary modifications.

\begin{lemma}
Suppose $\Xi$ and $\Theta$ are admissible decompositions with $\Xi\preceq \Theta$.  If one decomposition is extremal, then the other is as well.
\end{lemma}

Combining our results so far in this subsection, we obtain the following tool for proving results on generic $\mu$-stability of triples.

\begin{proposition}\label{prop-minimalReduction}
Suppose $\Xi$ is a minimal admissible decomposition and that $\Xi$ is generically $\mu$-stable.  Then any $\Theta$ which lies above $\Xi$ is also generically $\mu$-stable.  
\end{proposition}

\section{Stability of small rank extremal triples}\label{sec-smallRank}

The goal of this subsection is to prove the following theorem.

\begin{theorem}\label{thm-slopeCloseStable}
Let $\Xi = (\xi',\xi,\xi'')$ be an extremal triple with $\rk(\xi)\leq 6$.  Then $\Xi$ is generically $\mu$-stable.
\end{theorem}

By Proposition \ref{prop-minimalReduction}, we only need to consider cases where $\Xi$ is minimal.  We also assume $\Xi$ is torsion-free and defer to \S\ref{ssec-torsion} for the torsion case.  By twisting, we may assume $0<\mu(\xi)\leq 1$.  After these reductions, there are a relatively small number of triples to consider, which we list in Table \ref{table-slopeClose}.  For each triple, we also indicate the strategy we will use to prove the triple is generically $\mu$-stable.

\begin{center}
\renewcommand*{\arraystretch}{1.3}
\begin{longtable}{cccccccccccc}
\caption[]{The minimal, extremal, torsion-free triples $\Xi = (\xi',\xi,\xi'') = ((r',\mu',\Delta'),(r,\mu,\Delta),(r'',\mu'',\Delta''))$ which must be considered in Theorem \ref{thm-slopeCloseStable}.}\label{table-slopeClose}\\
\toprule 
$\xi'$ & $\xi$ & $\xi''$ && Strategy &$\qquad$ & $\xi'$ & $\xi$ & $\xi''$ && Strategy\\\midrule
\endfirsthead
\multicolumn{12}{l}{{\small \it continued from previous page}}\\
\toprule
$\xi'$ & $\xi$ & $\xi''$ && Strategy &$\qquad$ & $\xi'$ & $\xi$ & $\xi''$ && Strategy\\\midrule \endhead
\bottomrule \multicolumn{12}{r}{{\small \it continued on next page}} \\ \endfoot
\bottomrule
\endlastfoot
$(1,0,0)$ & $(2,\frac{1}{2},\frac{3}{8})$ & $(1,1,1)$ && Coprime &&$(2,\frac{1}{2},\frac{3}{8})$&$(5,\frac{3}{5},\frac{12}{25})$&$(3,\frac{2}{3},\frac{5}{9})$ &&Coprime\\
$(1,0,0)$ & $(3,\frac{1}{3},\frac{5}{9})$ & $(2,\frac{1}{2},\frac{7}{8})$ && Coprime &&$(4,\frac{3}{4},\frac{21}{32})$ & $(5,\frac{4}{5},\frac{18}{25})$ & (1,1,1) && Coprime\\
$(2,\frac{1}{2},\frac{3}{8})$ & $(3,\frac{2}{3},\frac{5}{9})$ & $(1,1,1)$ && Coprime &&$(1,0,0)$ & $(6,\frac{1}{6},\frac{55}{72})$ & $(5,\frac{1}{5},\frac{23}{25})$ && Coprime\\
$(1,0,0)$ & $(4,\frac{1}{4},\frac{21}{32})$ & $(3,\frac{1}{3},\frac{8}{9})$ && Coprime && $(4,\frac{1}{4},\frac{21}{32})$ & $(6,\frac{1}{3},\frac{5}{9})$ & $(2,\frac{1}{2},\frac{3}{8})$ && Complete\\
$(3,\frac{1}{3},\frac{5}{9})$ & $(4,\frac{1}{2},\frac{5}{8})$ & $(1,1,1)$ && Complete && $(5,\frac{2}{5},\frac{12}{25})$ & $(6,\frac{1}{2},\frac{17}{24})$ & $(1,1,2)$ && Prop. \ref{prop-rank6adHoc}\\
$(3,\frac{2}{3},\frac{5}{9})$ & $(4,\frac{3}{4},\frac{21}{32})$ & $(1,1,1)$ && Coprime && $(5,\frac{3}{5},\frac{12}{25})$ & $(6,\frac{2}{3},\frac{5}{9})$ & $(1,1,1)$ && Complete\\
$(1,0,0)$ & $(5,\frac{1}{5},\frac{18}{25})$ & $(4,\frac{1}{4},\frac{29}{32})$ && Coprime && $(5,\frac{4}{5},\frac{18}{25})$ & $(6,\frac{5}{6},\frac{55}{72})$ & $(1,1,1)$ && Coprime\\
$(3,\frac{1}{3},\frac{5}{9})$ &$(5,\frac{2}{5},\frac{12}{25})$ & $(2,\frac{1}{2},\frac{3}{8})$ && Coprime \\
\end{longtable}
\end{center}

Observing that $\xi''$ is always stable in Table \ref{table-slopeClose} completes the proof of Lemma \ref{lem-extremalExist} as promised.  The triples labelled ``Coprime'' are all generically $\mu$-stable by Corollary \ref{cor-slopeStable}.  We turn next to the triples labelled ``Complete.''

\begin{definition}
An admissible decomposition $\Xi$ is called \emph{complete} if the general $E\in M(\xi)$ can be expressed as an extension $$0\to F \to E \to Q\to 0$$ with $F\in M(\xi')$ and $Q\in M(\xi'')$.
\end{definition}

\begin{remark} Suppose $\Xi$ is admissible and generically semistable.  Recall the universal extension sheaf $E_s/S$ discussed preceding Definition \ref{def-stableTriple}.  If $U\subset S$ is the open subset parameterizing semistable sheaves, then $\Xi$ is complete if and only if the moduli map $U\to M(\xi)$ is dominant.  By generic smoothness, $E_s/U$ is a complete family of semistable sheaves over a potentially smaller dense open subset.
\end{remark}

If $\xi$ is stable, then the general sheaf in $M(\xi)$ is $\mu$-stable by a result of Dr\'{e}zet and Le Potier \cite[4.12]{DLP}.  Thus if $\Xi$ is complete, then $\Xi$ is generically $\mu$-stable.

\begin{proposition}\label{prop-completeTriples}
Let $\Xi$ be one of the three triples in Table \ref{table-slopeClose} labelled ``Complete.'' Then $\Xi$ is complete, and in particular generically $\mu$-stable.
\end{proposition}
\begin{proof}
First suppose $\Xi=(\xi',\xi,\xi'')$ is one of $$((3,\tfrac{1}{3},\tfrac{5}{9}),(4,\tfrac{1}{2},\tfrac{5}{8}),(1,1,1)) \qquad \textrm{or} \qquad ((4,\tfrac{1}{4},\tfrac{21}{32}),(6,\tfrac{1}{3},\tfrac{5}{9}),(2,\tfrac{1}{2},\tfrac{3}{8})).$$  Let $E$ be a $\mu$-stable sheaf of character $\xi$, and let $Q\in M(\xi'')$ be semistable.  We have $\chi(E,Q)>0$ in either case, which implies $\hom(E,Q)>0$ by stability.  Pick a nonzero homomorphism $f:E\to Q$, and let $R\subset Q$ be the image of $f$.   By stability considerations, $R$ must have the same rank and slope as $Q$, and $\Delta(R)\geq \Delta(Q)$.  Letting $F\subset E$ be the kernel of $f$, we find that $\rk(F) = \rk(\xi')$, $\mu(F)=\mu(\xi')$, and $\Delta(F)\leq \Delta(\xi')$, with equality if and only if $f$ is surjective.  Furthermore, $F$ is $\mu$-semistable.  Indeed, if there is a subsheaf $G\subset F$ with $\mu(G)>\mu(F)$, then $\mu(G)\geq \mu(E)$ by slope-closeness \ref{cond-slopeClose}, so $G\subset E$ violates $\mu$-stability of $E$.  Then discriminant minimality \ref{cond-discMinimal} forces $\ch F = \xi'$. Furthermore, since $\rk(\xi')$ and $c_1(\xi')$ are coprime, $F$ is actually semistable.  Thus $E$ is expressed as an extension $$0\to F\to E\to Q\to 0$$ of semistable sheaves as required.

For the final triple $((5,\frac{3}{5},\frac{12}{25}),(6,\frac{2}{3},\frac{5}{9}),(1,1,1))$ a slight modification to the previous argument is needed.  Fix a $\mu$-stable sheaf $E$ of character $\xi$.  This time $\chi(\xi,\xi'')=0$, so the expectation is that if $Q\in M(\xi'')$ is general, then there is no nonzero map $E\to Q$.  Consider the locus $$D_E = \{Q\in M(\xi''):\hom(E,Q)\neq 0\}.$$ Then either $D_E = M(\xi'')$ or $D_E$ is an effective divisor, in which case we can compute its class to show $D_E$ is nonempty.  Either way, there is some $Q\in M(\xi'')$ which admits a nonzero homomorphism $E\to Q$.  The argument can now proceed as in the previous cases.
\end{proof}

The next proposition treats the last remaining case, completing the proof of Theorem \ref{thm-slopeCloseStable}.

\begin{proposition}\label{prop-rank6adHoc}
The triple $\Xi =(\xi',\xi,\xi'')=((5,\frac{2}{5},\frac{12}{25}),(6,\frac{1}{2},\frac{17}{24}), (1,1,2))$ is generically $\mu$-stable
\end{proposition}
\begin{proof}
Observe that $\xi'$ is the Chern character of the exceptional bundle $F = E_{2/5}$ and $\xi''$ is the Chern character of an ideal sheaf $Q=I_Z(1)$, where $Z$ has degree $2$.  Let $Q_{s''}/M(\xi'')$ be the universal family.  Then the projective bundle $S$ over $M(\xi'')$ with fibers $\P\Ext^1(Q_{s''},F)$ is smooth and irreducible of dimension $$\dim S = \dim M(\xi'')-\chi(\xi'',\xi')-1=14,$$ and there is a universal extension $E_s/S$.  Every $E_s$ is $\mu$-semistable by Proposition \ref{prop-slopeSemistable}.

A simple computation shows $\Hom(F,Q_{s''})=0$ for every $Q_{s''}$.  
If $E$ is any sheaf which sits as an extension $$0\to F\to E\to Q_{s''}\to 0,$$ then we apply $\Hom(F,-)$ to see $\Hom(F,E) \cong \Hom(F,F) = \C$.  Thus the homomorphism $F\to E$ is unique up to scalars, the sheaf $Q_{s''}$ is determined as the cokernel, and since $Q_{s''}$ is simple the corresponding extension class in $\Ext^1(Q_{s''},F)$ is determined up to scalars.  We find that distinct points of $S$ parameterize non-isomorphic sheaves.  A straightforward computation further shows that the Kodaira-Spencer map $T_sS \to \Ext^1(E_s,E_s)$ is injective for every $s\in S$.

We now proceed to show that the general $E_s$ also satisfies stronger notions of stability.

\emph{Step 1: the general $E_s$ is semistable}.  If $E_s$ is not semistable, it has a Harder-Narasimhan filtration of length $\ell\geq 2$, and all factors have slope $\frac{1}{2}$.  For each potential set of numerical invariants of a Harder-Narasimhan filtration, we check that the corresponding Shatz stratum of $s\in S$ such that the Harder-Narasimhan filtration has that form has positive codimension.  

There are only a handful of potential numerical invariants of the filtration.  A non-semistable $E_s$ has a semistable subsheaf $G$ with $\mu(G) = \mu(E) = \frac{1}{2}$ and $\Delta(G) < \Delta(E) = \frac{17}{24}$. Then the Chern character of $G$ must be one of 
\begin{equation}\tag{$\ast$} (2,\tfrac{1}{2},\tfrac{3}{8}), 
\qquad (4,\tfrac{1}{2},\tfrac{3}{8}),
 \qquad \textrm{or} \qquad (4,\tfrac{1}{2},\tfrac{5}{8}).\end{equation}
We can rule out the first two cases immediately by an ad hoc argument.  In either of these cases $E$ has a subsheaf isomorphic to $T_{\P^2}(-1)$.  Then there is a sequence $$0\to T_{\P^2}(-1)\to E\to R\to 0.$$ Applying $\Hom(F,-)$, we see $\Hom(F,T_{\P^2}(-1))$ injects into $\Hom(F,E)=\C$.  But $\chi(F,T_{\P^2}(-1))=3$, so this is absurd.
 
 Thus the only Shatz stratum we must consider is the locus of sheaves with a filtration $$0 \subset G_1 \subset G_2 = E_s$$ having $\ch \gr_1 = \zeta_1 := (4,\frac{1}{2},\frac{5}{8})$ and $\ch \gr_2 := \zeta_2 = (2,\frac{1}{2},\frac{7}{8})$.  Let $$\Sigma = \Flag(E/S;\zeta_1,\zeta_2) \xrightarrow{\pi} S$$ be the relative flag variety parameterizing sheaves with a filtration of this form.  By the uniqueness of the Harder-Narasimhan filtration, $\pi$ is injective, and its image is the Shatz stratum.  The differential of $\pi$ at a point $t = (s,G_1)\in \Sigma$ can be analyzed via the exact sequence $$0 \to \Ext^0_+(E_s,E_s)\to T_t \Sigma\xrightarrow{T_t\pi} T_s S\xrightarrow{\omega_+} \Ext_+^1(E_s,E_s).$$ We have $\Ext_+^0(E_s,E_s)=0$ by \cite[Proposition 15.3.3]{LePotierLectures}, so $T_t\pi$ is injective and $\pi$ is an immersion.  The codimension of the Shatz stratum near $s$ is at least  $\rk \omega_+$.

The map $\omega_+$ is the composition $T_sS\to \Ext^1(E_s,E_s)\to \Ext^1_+(E_s,E_s)$ of the Kodaira-Spencer map with the canonical map from the long exact sequence of $\Ext_{\pm}$.  The Kodaira-Spencer map is injective, and $\Ext^1(E_s,E_s)\to \Ext_+^1(E_s,E_s)$ is surjective since $\Ext_-^2(E_s,E_s)=0$.  We have $$\dim T_sS = 14, \qquad \ext^1(E_s,E_s)=16, \qquad \textrm{and} \qquad  \ext^1_+(E_s,E_s)=-\chi(\gr_1,\gr_2) = 4,$$ so we conclude $\rk \omega_+\geq 2$.  Therefore the Shatz stratum is a proper subvariety of $S$.   We conclude $\Xi$ is generically semistable.

\emph{Step 2: the general $E_s$ is $\mu$-stable}.  Note that a semistable sheaf in $M(\xi)$ is automatically stable.  Then the moduli map $S\to M^s(\xi)$ is injective, and its image has codimension $2$ in $M^s(\xi)$.

If a sheaf $E\in M^s(\xi)$ is not $\mu$-stable, then there is a filtration $$0\subset G_1 \subset G_2 = E$$ such that the quotients $\gr_i$ are semistable of slope $\frac{1}{2}$ and $\Delta(\gr_1) > \Delta(E) > \Delta(\gr_2)$ (see the proof of \cite[Theorem 4.11]{DLP}).  Then as in the previous step $\zeta_2 = \ch(\gr_2)$ is one of the characters $(\ast)$, and $\zeta_1 = \ch(\gr_1)$ is determined by $\zeta_2$. For each of the three possible filtrations, the Shatz stratum in $M^s(\xi)$ of sheaves with a filtration of the given form has codimension at least $\ext^1_+(E,E) = -\chi(\gr_1,\gr_2)$.

When $\zeta_2 = (4,\frac{1}{2},\frac{3}{8})$ we compute $-\chi(\gr_1,\gr_2) = 6$, and when $\zeta_2 = (4,\frac{1}{2},\frac{5}{8})$ we have $-\chi(\gr_1,\gr_2) = 4$.  In particular, the corresponding Shatz strata have codimension bigger than $2$.  On the other hand, for $\zeta_2 = (2,\frac{1}{2},\frac{3}{8})$ we only find the stratum has codimension at least $2$, and it is a priori possible that it contains the image of $S\to M^s(\xi)$.

To get around this final problem, we must show that the general sheaf $E_s$ parameterized by $S$ does not admit a nonzero map $E_s\to T_{\P^2}(-1)$.  This can be done by an explicit calculation.  Put $Q = I_Z(1)$, where $Z = V(x,y^2)$.  By stability, $\Hom(Q,T(-1))=0$, so there is an exact sequence $$\xymatrix{ 0 \ar[r]& \Hom(E_s,T_{\P^2}(-1))\ar[r] & \Hom(F,T_{\P^2}(-1))\ar[r]^{f} \ar@{=}[d] & \Ext^1(Q,T_{\P^2}(-1))\ar@{=}[d]\\ &&\C^3&\C^{4}&}$$ and we must see $f$ is injective.  The map $f$ is the contraction of the canonical map  $$\Ext^1(Q,F) \te \Hom(F,T_{\P^2}(-1))\to \Ext^1(Q,T_{\P^2}(-1))$$ corresponding to the extension class of $E$ in $\Ext^1(Q,F)$.  This canonical map can be explicitly computed using the standard resolutions $$\xymatrix@R=1mm{ 0 \ar[r] &\OO_{\P^2}(-2)\ar[r] &\OO_{\P^2}^6 \ar[r] &F \ar[r] &0\\ 0\ar[r] &\OO_{\P^2}(-2) \ar[r]&\OO_{\P^2}(-1)\oplus \OO_{\P^2} \ar[r] & Q\ar[r]& 0\\ 0\ar[r]& \OO_{\P^2}(-1)\ar[r]& \OO_{\P^2}^3 \ar[r]& T_{\P^2}(-1)\ar[r]& 0}$$ with the special form of $Q$ simplifying the calculation.  Injectivity of $f$ for a general $E_s$ follows easily from this computation.  
\end{proof}

\section{Curves of extensions}\label{sec-curves} 

\subsection{General results} Let $F$ and $Q$ be sheaves, and suppose the general extension $E$ of $Q$ by $F$ is semistable of Chern character $\xi$.  In this section, we study the moduli map $$\P \Ext^1(Q,F)\dashrightarrow M(\xi).$$  In particular, we would like to be able to show this map is nonconstant.  

\begin{definition}
Let $\Xi$ be a generically semistable admissible decomposition.  We say $\Xi$ \emph{gives curves} if for a general $F\in M(\xi')$ and $Q\in M(\xi'')$, the map $\P \Ext^1(Q,F)\dashrightarrow M(\xi)$ is nonconstant.
\end{definition}

There are three essential ways that  $\Xi$ could fail to give curves.
\begin{enumerate}
\item If $-\chi(\xi'',\xi') = 1$, then $\P\Ext^1(Q,F)$ is a point.
\item Sheaves parameterized by $\P \Ext^1(Q,F)$ might all be strictly semistable and $S$-equivalent.  
\item The sheaves parameterized by $\P\Ext^1(Q,F)$ might all be isomorphic.
\end{enumerate}
Possibility (1) is easy to check for any given triple.  If $\Xi$ is generically stable, then possibility (2) cannot arise when $F$ and $Q$ are general, so this is also easy to rule out.  The third case requires the most work to deal with.

\begin{lemma}\label{lem-nonIsoExtensions}
Let $F$ and $Q$ be simple sheaves with $\Hom(F,Q) = 0$.  Then distinct points of $\P\Ext^1(Q,F)$ parameterize nonisomorphic sheaves.
\end{lemma}
\begin{proof}
Suppose $E$ is a sheaf which can be realized as an extension $$0\to F\to E\to Q\to 0.$$ Since $F$ is simple and $\Hom(F,Q) = 0$ we find $\hom(F,E) = 1$.  Similarly, since $Q$ is simple and $\Hom(F,Q)=0$, we have $\hom(E,Q)=1$.  This means that the corresponding class in $\P\Ext^1(Q,F)$ depends only on the isomorphism class of $E$.
\end{proof}

The lemma gives us a simple criterion for proving a triple $\Xi$ gives curves.

\begin{proposition}\label{prop-curveCriterion}
Let $\Xi=(\xi',\xi,\xi'')$ be an admissible, generically stable triple, and assume $\xi'$ is stable.  Suppose either 
\begin{enumerate}
\item $\Xi$ is not minimal, or
\item $-\chi(\xi'',\xi') \geq 2$.
\end{enumerate}
If $\Hom(F,Q)=0$ for a general $F\in M(\xi')$ and $Q\in M(\xi'')$, then $\Xi$ gives curves.
\end{proposition}
\begin{proof}
If $\Xi$ is not minimal, then $-\chi(\xi'',\xi')\geq 2$ holds automatically.  Indeed, since $\Xi$ is not minimal it is an elementary modification of another admissible triple $\Theta = (\theta',\theta,\theta'')$.  Then $\chi(\xi'',\xi')<\chi(\theta'',\theta')<0$ by the Riemann-Roch formula and Lemma \ref{existenceOfExtensions}.

Since $\xi'$ and $\xi''$ are stable, we can choose stable sheaves $F\in M(\xi')$ and $Q\in M(\xi'')$ such that $\Hom(F,Q) = 0$ and the general extension of $Q$ by $F$ is stable.  By Lemma \ref{lem-nonIsoExtensions}, $\Xi$ gives curves. 
\end{proof}

We also observe that elementary modifications behave well with respect to the notion of giving curves.

\begin{lemma}\label{lem-curveBump}
Suppose $\Xi$ is admissible and generically $\mu$-stable.  If $\Xi \preceq \Theta$ and $\Xi$ gives curves, then $\Theta$ gives curves. 
\end{lemma}
\begin{proof}
Suppose $\Theta$ is obtained from $\Xi$ by a single elementary modification.  Let $F\in M(\xi')$ and $Q\in M(\xi'')$ be general.  Take $U \subset \P\Ext^1(Q,F)$ to be the dense open subset parameterizing $\mu$-stable sheaves $E$ which are locally free at a general fixed point $p\in \Supp Q$.  Let $Q\to \OO_p$ be a surjective homomorphism.  Given any extension $$0\to F \to E \to Q \to 0$$ corresponding to a point of $U$, we get an exact sequence of compatible elementary modifications $$0\to F\to E'\to Q'\to 0$$ as in Definition \ref{def-elementaryModSequence}.  As $E$ can be recovered from $E'$ and the map $U\to M(\xi)$ is nonconstant, we conclude that the map $\P\Ext^1(Q',F)\dashrightarrow M(\theta)$ is nonconstant.  Thus $\Theta$ gives curves.
\end{proof}

\subsection{Curves from coprime triples with large discriminant} Our next result provides the dual curves we will need to prove Theorem \ref{thm-asymptotic}.

\begin{theorem}\label{thm-curves}
Let $\Xi = (\xi',\xi,\xi'')$ be a coprime extremal triple, and suppose $\Delta(\xi)$ is sufficiently large, depending on $\rk(\xi)$ and $\mu(\xi)$.  Then $\Xi$ gives curves.
\end{theorem} 
\begin{proof}
The triple $\Xi$ is not minimal since $\Delta(\xi)$ is large.  By Proposition \ref{prop-curveCriterion}, we only need to show that if $F\in M(\xi')$ and $Q\in M(\xi'')$ are general and $\Delta(\xi)$ is sufficiently large, then $\Hom(F,Q) = 0$.  Fix a general $F\in M(\xi')$ an $Q\in M(\xi'')$.  If $\Hom(F,Q)\neq 0$, choose a nonzero homomorphism $F\to Q$.  We can find a surjective homomorphism $Q\to \OO_p$ such that $F\to Q\to \OO_p$ is also surjective.  Then applying $\Hom(F,-)$ to the elementary modification sequence $$0\to Q'\to Q\to \OO_p\to 0$$ we find that $\Hom(F,Q')$ is a proper subspace of $\Hom(F,Q)$.  Repeating this process, we can find some $\Theta \succeq \Xi$ such that $\Hom(F,Q)=0$ for general $F\in M(\theta')$ and $Q\in M(\theta'')$.  Then $\Theta$ gives curves, and by Lemma \ref{lem-curveBump} any $\Lambda \succeq \Theta$ also gives curves.
\end{proof}

\subsection{Curves from small rank triples} We now discuss extremal curves in the moduli space $M(\xi)$ when the rank is small.  For all but a handful of characters $\xi$ we can apply the next theorem.

\begin{theorem}\label{thm-smallRankCurves}\label{thm-curvessmall}
Let $\Xi = (\xi',\xi,\xi'')$ be an extremal  triple with $\rk(\xi)\leq 6$.  Suppose $\chi(\xi',\xi'')\leq 0$.  Then $\Xi$ gives curves.
\end{theorem}

\begin{proof} As with the proof of Theorem \ref{thm-slopeCloseStable}, we assume $0<\mu(\xi)\leq 1$.  By Lemma \ref{lem-curveBump}, it is enough to consider triples $\Xi$ such that any admissible $\Theta$ with $\Theta \prec \Xi$ has $\chi(\theta',\theta'')>0$.  We also assume $\Xi$ is torsion-free, and handle the torsion case in \S\ref{ssec-torsion}.
We list the relevant triples together with $\chi(\xi',\xi'')$ in Table \ref{table-curves}.

\begin{center}
\renewcommand*{\arraystretch}{1.3}
\begin{longtable}{cccccccccccc}
\caption[]{Triples to be considered for the proof of Theorem \ref{thm-smallRankCurves}.}\label{table-curves}\\
\toprule 
$\xi'$ & $\xi$ & $\xi''$ && $\chi(\xi',\xi'')$ &$\qquad$ & $\xi'$ & $\xi$ & $\xi''$ && $\chi(\xi',\xi'')$\\\midrule
\endfirsthead
\multicolumn{12}{l}{{\small \it continued from previous page}}\\
\toprule
$\xi'$ & $\xi$ & $\xi''$ && $\chi(\xi',\xi'')$ &$\qquad$ & $\xi'$ & $\xi$ & $\xi''$ && $\chi(\xi',\xi'')$\\\midrule \endhead
\bottomrule \multicolumn{12}{r}{{\small \it continued on next page}} \\ \endfoot
\bottomrule
\endlastfoot
$(1,0,0)$ & $(2,\frac{1}{2},\frac{11}{8})$ & $(1,1,3)$ && $0$ &&$(2,\frac{1}{2},\frac{3}{8})$&$(5,\frac{3}{5},\frac{17}{25})$&$(3,\frac{2}{3},\frac{8}{9})$ &&0\\
$(1,0,0)$ & $(3,\frac{1}{3},\frac{11}{9})$ & $(2,\frac{1}{2},\frac{15}{8})$ && $0$ &&$(4,\frac{3}{4},\frac{21}{32})$ & $(5,\frac{4}{5},\frac{18}{25})$ & (1,1,1) && $-1$\\
$(2,\frac{1}{2},\frac{3}{8})$ & $(3,\frac{2}{3},\frac{8}{9})$ & $(1,1,2)$ && $-1$ &&$(1,0,0)$ & $(6,\frac{1}{6},\frac{79}{72})$ & $(5,\frac{1}{5},\frac{33}{25})$ && 0\\
$(1,0,0)$ & $(4,\frac{1}{4},\frac{37}{32})$ & $(3,\frac{1}{3},\frac{14}{9})$ && 0 && $(4,\frac{1}{4},\frac{21}{32})$ & $(6,\frac{1}{3},\frac{13}{18})$ & $(2,\frac{1}{2},\frac{7}{8})$ && $-1$\\
$(3,\frac{1}{3},\frac{5}{9})$ & $(4,\frac{1}{2},\frac{7}{8})$ & $(1,1,2)$ && $-1$ && $(5,\frac{2}{5},\frac{12}{25})$ & $(6,\frac{1}{2},\frac{17}{24})$ & $(1,1,2)$ && $-2$\\
$(3,\frac{2}{3},\frac{5}{9})$ & $(4,\frac{3}{4},\frac{21}{32})$ & $(1,1,1)$ && 0 && $(5,\frac{3}{5},\frac{12}{25})$ & $(6,\frac{2}{3},\frac{13}{18})$ & $(1,1,2)$ && $-4$\\
$(1,0,0)$ & $(5,\frac{1}{5},\frac{28}{25})$ & $(4,\frac{1}{4},\frac{45}{32})$ && 0 && $(5,\frac{4}{5},\frac{18}{25})$ & $(6,\frac{5}{6},\frac{55}{72})$ & $(1,1,1)$ && $-2$\\
$(3,\frac{1}{3},\frac{5}{9})$ &$(5,\frac{2}{5},\frac{17}{25})$ & $(2,\frac{1}{2},\frac{7}{8})$ && $-1$ \\
\end{longtable}
\end{center}

Every triple in the table satisfies $\chi(\xi'',\xi')\leq -2$.  To apply Proposition \ref{prop-curveCriterion}, we need to see that if $F\in M(\xi')$ and $Q\in M(\xi'')$ are general, then $\Hom(F,Q)= 0$.  This can easily be checked case by case via standard sequences or by Macaulay2. We omit the details.\end{proof}

\subsection{Sporadic small rank triples}\label{ssec-sporadic} Here we discuss the few sporadic Chern characters not addressed by Theorem \ref{thm-smallRankCurves}.  Suppose $\xi$ has positive height and $\rk(\xi)\leq 6$, and let $\Xi = (\xi',\xi,\xi'')$ be extremal.  If $\Xi$ is torsion-free and $\chi(\xi',\xi'')>0$, then $\xi' = (1,0,0)$ and $$\xi\in \{(2,\tfrac{1}{2},\tfrac{7}{8}),(3,\tfrac{1}{3},\tfrac{8}{9}),(4,\tfrac{1}{4},\tfrac{29}{32}),(5,\tfrac{1}{5},\tfrac{23}{25}),(6,\tfrac{1}{6},\tfrac{67}{72})\}.$$  That is, $$\xi = (r,\tfrac{1}{r},P(-\tfrac{1}{r})+\tfrac{1}{r})$$ for some $r$ with $2\leq r \leq 6$.  In this case, we have $\chi(\xi',\xi) = 2$, so the general sheaf $E\in M(\xi)$ admits \emph{two} maps from $\OO_{\P^2}$.  Along the wall $W(\Xi)$, the destabilizing subobject of $E$ should therefore be $\OO_{\P^2}^2$ instead of $\OO_{\P^2}$.  Furthermore, assuming there are sheaves $E\in M(\xi)$ which are Bridgeland stable just outside $W(\Xi)$, the wall $W(\Xi)$ must be the collapsing wall.  We will show in the next section that $W(\Xi)$ is actually the Gieseker wall.  This implies that the primary edges of the ample and effective cones of divisors on $M(\xi)$ coincide.  Our study of the effective cone in \cite{CoskunHuizengaWoolf} easily implies the next result.

\begin{proposition}\label{prop-sporadic}
Let $r\geq 2$, and let $\Xi=(\xi',\xi,\xi'')$ be the admissible decomposition with $\xi' = (2,0,0)$ and $\xi = (r,\tfrac{1}{r},P(-\tfrac{1}{r})+\tfrac{1}{r})$.  Then $\Xi$ is complete, and it gives curves.
\end{proposition}

While we had to modify the original extremal triple $\Xi$ in order to make it give curves, note that the wall $W(\Xi)$ is unchanged by this modification.  When showing that $W(\Xi)$ is the largest wall in \S\ref{sec-ample}, we will not need to handle these cases separately.

In the case of the Chern character $\xi = (6,\frac{1}{3},\frac{13}{18})$, Theorem \ref{thm-smallRankCurves} shows that the corresponding extremal triple $\Xi$ gives curves.  However, the wall $W(\Xi)$ is actually empty.  In this case $\chi(\OO_{\P^2},E) = 5$ for $E\in M(\xi)$ and the Chern character $\xi' = (5,0,0)$ will correspond to the primary edges of both the effective and ample cones.  Again, curves dual to this edge of the effective cone are given by \cite{CoskunHuizengaWoolf}.

\begin{proposition}\label{prop-sporadic2}
The admissible decomposition $\Xi = ((5,0,0),(6,\frac{1}{3},\frac{13}{18}),(1,2,6))$ is complete and gives curves.
\end{proposition}

\subsection{Torsion triples}\label{ssec-torsion} 

Let $\Xi=(\xi',\xi,\xi'')$ be an extremal torsion triple, and let $r=\rk(\xi) = \rk(\xi')$.  We have $\mu(\xi)-\mu(\xi') = \frac{1}{r}$ by the slope-closeness condition \ref{cond-slopeClose}.  Write $$\mu(\xi') = \frac{a}{b} \qquad \textrm{and}\qquad \mu(\xi) = \frac{c}{d}$$ in lowest terms.  The numbers $\mu(\xi')$ and $\mu(\xi)$ are consecutive terms in the \emph{Farey sequence} of order $r$, so $\mu(\xi) - \mu(\xi') = \frac{1}{bd}$ and we deduce that $bd=r$.  Also, the \emph{mediant} $$\frac{a+c}{b+d}$$ must have denominator $b+d$ larger than $r$.  The two conditions $bd = r$ and $b+d>r$ together imply that either $b=1$ or $d=1$.  That is, either $\mu(\xi')$ or $\mu(\xi)$ is an integer.

If $\mu(\xi')$ is an integer, then by discriminant minimality \ref{cond-discMinimal}  and rank minimality \ref{cond-rankMinimal} we have $\xi' = (1,\mu(\xi'),0)$.  Thus $r=1$, and in every case $\mu(\xi)$ is also an integer.

We may assume $\mu(\xi) = 1$.  Consider the triple $\Xi$ with $$\xi' = (r,1-\tfrac{1}{r},P(-\tfrac{1}{r})) \qquad \textrm{and} \qquad \xi = (r,1,1).$$  The character $\xi''$ is then $\ch \OO_L$, where $L\subset \P^2$ is a line.  We have $\chi(\xi,\xi'')=0$, and $\Xi$ is complete (hence generically $\mu$-stable) by a similar argument to Proposition \ref{prop-completeTriples}. If $r\geq 2$, then  $$\dim M(\xi')+\dim M(\xi'')=(r^2-3r+2)+2< r^2+1 = \dim M(\xi),$$ so $\Xi$ must give curves. When $r=1$, for any $p\in L$ there is a sequence $$0\to \OO_{\P^2}\to I_p(1) \to \OO_L\to 0,$$ so $\Xi$ gives curves when $r=1$ as well.  Applying elementary modifications, we conclude our discussion with the following.

\begin{proposition}
Let $\Xi = (\xi',\xi,\xi'')$ be a torsion extremal triple.  If $\xi$ is not the Chern character of a line bundle, then $\Xi$ is generically $\mu$-stable and gives curves.
\end{proposition}

\section{The ample cone}\label{sec-ample}

\subsection{Notation} We begin by fixing notation for the rest of the paper.  Let $\xi$ be a stable Chern character of positive height.  We assume one of the following three hypotheses hold:
 \begin{enumerate}[label=(H\arabic*)]
\item \label{hyp-asymptotic} $\rk(\xi)$ and $c_1(\xi)$ are coprime and $\Delta(\xi)$ is sufficiently large, 
\item \label{hyp-smallRank} $\rk(\xi) \leq 6$ and $\xi$ is not a twist of $(6,\frac{1}{3},\frac{13}{18})$, or
\item \label{hyp-exceptionalCase} $\xi = (6,\frac{1}{3},\frac{13}{18})$. 
\end{enumerate}
Suppose we are in case \ref{hyp-asymptotic} or \ref{hyp-smallRank}.  There is an extremal triple $\Xi = (\xi',\xi,\xi'')$.  Either $\Xi$ gives curves or we are in one of the cases of Proposition \ref{prop-sporadic}, in which case there is a decomposition of $\xi$ which gives curves and has the same corresponding wall.  As discussed in \S\ref{ssec-BayerMacriPlan}, to show the primary edge of the ample cone corresponds to $W(\Xi)$ it will be enough to show that $W(\Xi)$ is the Gieseker wall $W_{\max}$.  Note that $W_{\max}$ cannot be strictly nested inside $W(\Xi)$, since then by our work so far there are sheaves $E\in M(\xi)$ destabilized along $W(\Xi)$.  We must show $W_{\max}$ is not larger than $W(\Xi)$.

Let $E\in M(\xi)$ be a sheaf which is destabilized along some wall.  For any $(s,t)$ on the wall we have an exact sequence $$0\to F\to E\to Q\to 0$$ of $\sigma_{s,t}$-semistable objects of the same slope, where the sequence is exact in any of the corresponding categories $\cA_s$ along the wall.  Above the wall, $\mu_{s,t}(F)<\mu_{s,t}(E)$, and below the wall the inequality is reversed.  Let $\Theta = (\theta',\theta,\theta'') = (\ch F,\ch E,\ch Q)$ be the corresponding decomposition of $\xi = \theta$, so that the wall is $W(\Theta)$.
Our job is to show that $W(\Theta)$ is no larger than $W(\Xi)$ by imposing numerical restrictions on $\theta'$.  We begin by imposing some easy restrictions on $\theta'$.

\begin{lemma}\label{lem-boundsTrivial}
The object $F$ is a nonzero torsion-free sheaf, so $\rk(F)\geq 1$.  We have $\mu(F) < \mu(E)$, and every Harder-Narasimhan factor of $F$ has slope at most $\mu(E)$.
\end{lemma}
\begin{proof}
Fix a category $\cA_s$ along $W(\Theta)$.  Taking cohomology sheaves of the destabilizing sequence of $E$, we get a long exact sequence $$0\to {\rm H}^{-1}(F)\to 0 \to {\rm H}^{-1}(Q) \to {\rm H}^0(F)\to {\rm H}^0(E)\to {\rm H}^0(Q)\to 0$$ since $E\in \cQ_s$.  Thus $F$ is a sheaf in $\cQ_s$.  We write $K = {\rm H}^{-1}(Q)$ and $C = {\rm H}^{0}(Q)$, so $K\in \cF_s$,  $C\in \cQ_s$, and we have an exact sequence of sheaves $$0\to K\to F \to E \to C\to 0.$$
Since $E$ is torsion-free, the torsion subsheaf of $F$ is contained in $K$.  Since $K\in \cF_s$ is torsion-free,  we conclude $F$ is torsion-free.  Clearly also $F$ is nonzero, for otherwise $(\theta',\theta,\theta'')$ wouldn't span a $2$-plane in $K(\P^2)$.  We conclude $\rk(F)\geq 1$.

Let $$\{0\} \subset F_1\subset \cdots \subset F_\ell = F$$ be the Harder-Narasimhan filtration of $F$.  If $\mu(F_1) > \mu(E)$ then $F_1\to E$ is zero and $F_1\subset K$.  Since $K\in \cF_s$ for any $s$ along $W(\Theta)$,  this is absurd.  Therefore $\mu(F_1)\leq \mu(E)$.  We can't have $\mu(F) = \mu(E)$ since then $W(\Theta)$ would be the vertical wall, so we conclude $\mu(F)<\mu(E)$.
\end{proof}

\subsection{Excluding higher rank walls} In this subsection, we bound the rank of $F$ under the assumption that $W(\Theta)$ is larger than $W(\Xi)$.  In general, there will be walls corresponding to ``higher rank'' subobjects.  We show that the Gieseker wall cannot correspond to such a subobject.

\begin{theorem}\label{thm-excludeHighRank}
Keep the notation and hypotheses from above.  
\begin{enumerate} \item Suppose hypothesis \ref{hyp-asymptotic} or \ref{hyp-smallRank} holds.   If $W(\Theta)$ is larger than $W(\Xi)$, then $1\leq \rk(\theta') \leq
 \rk(\xi)$.
 \item If $\xi = (6,\frac{1}{3},\frac{13}{18})$, then the same result holds for the decomposition $\Xi = (\xi',\xi,\xi'')$ with $\xi' = (5,0,0)$.
 \end{enumerate}
\end{theorem}

The next inequality is our main tool for proving the theorem.

\begin{proposition}\label{prop-highRank}
If $\rk(F)>\rk(E)$, then the radius $\rho_\Theta$ of $W(\Theta)$ satisfies $$\rho_\Theta^2 \leq \frac{\rk(E)^2}{2(\rk(E)+1)}\Delta(E).$$
\end{proposition}
\begin{proof}
Consider the exact sequence of sheaves $$0\to K^k\to F^f\to E^e\to C^c\to 0,$$ with the superscripts denoting the ranks of the sheaves.  Since $F$ is in the categories $\cQ_s$ along $W(\Theta)$, we have $$f(s_\Theta+\rho_\Theta)\leq f\mu(F) = c_1(F) = c_1(K)+c_1(E) - c_1(C)= k\mu(K)+e\mu(E)-c_1(C).$$

Next, since $K$ is nonzero and in $\cF_s$ along $W(\Theta)$, we have $\mu(K)\leq s_\Theta-\rho_\Theta$, and thus $$f(s_\Theta+\rho_\Theta)\leq k(s_\Theta-\rho_\Theta)+e\mu(E)-c_1(C).$$ Rearranging, $$(k+f)\rho_\Theta\leq (k-f)s_\Theta+e\mu(E) - c_1(C).$$ If $C$ is zero or torsion, then $k-f=-e$ and $c_1(C)\geq 0$, from which we get \begin{equation}\label{eqn1} (k+f)\rho_\Theta\leq (k-f)(s_\Theta-\mu(E))\end{equation} This inequality also holds if $C$ is not torsion.  In that case, we have $k-f=c-e$ and since $E$ is semistable $c_1(C) = c\mu(C) \geq c\mu(E)$, from which the inequality follows.

Both sides of Inequality (\ref{eqn1}) are positive, and squaring both sides gives $$(k+f)^2\rho_\Theta^2\leq (k-f)^2(\rho_{\Theta}^2+2\Delta(E)).$$ We conclude $$\rho_\Theta^2 \leq \frac{(k-f)^2}{2kf} \Delta(E).$$ This inequality is as weak as possible when the coefficient $(k-f)^2/(2kf)$ is maximized.  Viewing $e$ as fixed, $k$ and $f$ are integers satisfying $f\geq e+1$ and $f-e\leq k\leq f$.  It is easy to see the coefficient is maximized when $f = e+1$ and $k=1$, which corresponds to the  inequality we wanted to prove.
\end{proof}

\begin{proof}[Proof of Theorem \ref{thm-excludeHighRank}]
We recall $$\rho_\Xi^2 = \left(\frac{\mu(\xi')-\mu(\xi)}{2}-\frac{\Delta(\xi)-\Delta(\xi')}{\mu(\xi)-\mu(\xi')}\right)^2-2\Delta(\xi).$$ If we view $\rho_\Xi^2$ as a function of $\Delta(\xi)$, then it grows quadratically as $\Delta(\xi)$ increases.  Suppose $\Delta(\xi)$ is large enough so that $$\rho_\Xi^2 \geq \frac{\rk(\xi)^2}{2(\rk(\xi)+1)}\Delta(\xi).$$ Then if $\rk(\theta')>\rk(\xi)$, we have  $\rho^2_\Theta\leq \rho^2_\Xi$.  This proves the theorem if hypothesis \ref{hyp-asymptotic} holds.

Next suppose \ref{hyp-smallRank} holds, and write $\xi = (r,\mu,\Delta)$.  View $\Delta$ as variable, and consider the quadratic equation in $\Delta$  $$\rho_\Xi^2 = \frac{r^2}{2(r+1)}\Delta;$$ this equation depends only on $r$ and $\mu$.  Assuming this equation has roots, let $\Delta_1(r,\mu)$ be the larger of the two roots.  Then the theorem is true for $\xi$ if $\Delta\geq \Delta_1(r,\mu)$.   Let $\Delta_0(r,\mu)$ be the minimal discriminant of a rank $r$, slope $\mu$ sheaf satisfying \ref{hyp-smallRank}.  We record the values of $\Delta_0(r,\mu)$ and $\Delta_1(r,\mu)$ for all pairs $(r,\mu)$ with $1\leq r\leq 6$ and $0 < \mu \leq 1$ in Table \ref{table-discValues}.  For later use, we also record the value of the right endpoint $(x^+(r,\mu),0)$ of the wall $W(\Xi)$ corresponding to the character $(r,\mu,\Delta_0)$.

\begin{center}
\renewcommand*{\arraystretch}{1.3}
\begin{longtable}{ccccccccccccccccc}
\caption[]{Computation of $\Delta_0(r,\mu)$, $\Delta_1(r,\mu)$, and $x^+(r,\mu)$.}\label{table-discValues}\\
\toprule 
$r$ & $\mu$ & $\Delta_0$ & $\Delta_1 $ &$x^+$ &$\qquad$& $r$ & $\mu$ & $\Delta_0$ & $\Delta_1$ &$x^+$&$\qquad$&$r$ & $\mu$ & $\Delta_0$ & $\Delta_1$ &$x^+$\\\midrule
\endfirsthead
\multicolumn{17}{l}{{\small \it continued from previous page}}\\
\toprule
$r$ & $\mu$ & $\Delta_0$ & $\Delta_1 $ &$x^+$ &$\qquad$& $r$ & $\mu$ & $\Delta_0$ & $\Delta_1$ &$x^+$&$\qquad$&$r$ & $\mu$ & $\Delta_0$ & $\Delta_1$ &$x^+$\\\midrule \endhead
\bottomrule \multicolumn{17}{r}{{\small \it continued on next page}} \\ \endfoot
\bottomrule
\endlastfoot
1 & 1 & 2 & 1.00 & 0                           && 
4 & $\frac{1}{2}$ & $\frac{7}{8}$ & 0.81 & 0    &&
5 & 1 & $\frac{6}{5}$ & 1.13 &0.46                    \\
2 & $\frac{1}{2}$ & $\frac{7}{8}$ & 0.25 & 0   && 
4 & $\frac{3}{4}$ & $\frac{29}{32}$ & 0.67 & 0.53 &&           
6 & $\frac{1}{6}$ & $\frac{67}{72}$ & 0.03 & 0     \\
2 & 1 & $\frac{3}{2}$ & 1.11 & 0.30               && 
4 & 1 & $\frac{5}{4}$ & 1.13 & 0.44               &&
6 & $\frac{1}{3}$ & $\frac{8}{9}$ & 0.79 & 0     \\
3 & $\frac{1}{3}$ & $\frac{8}{9}$ & 0.11 & 0  &&
5 & $\frac{1}{5}$ & $\frac{23}{25}$ & 0.04 &0 &&
6 & $\frac{1}{2}$ & $\frac{17}{24}$ & 0.64 &0.17     \\
3 & $\frac{2}{3}$ & $\frac{8}{9}$ & 0.57 & 0.37  && 
5 & $\frac{2}{5}$ & $\frac{17}{25}$ & 0.65 &0 &&
6 & $\frac{2}{3}$ & $\frac{13}{18}$ & 0.58 & 0.46    \\
3 & 1 & $\frac{4}{3}$ & 1.13  &0.39              && 
5 & $\frac{3}{5}$ & $\frac{17}{25}$ & 0.48 &0.37 &&
6 & $\frac{5}{6}$ & $\frac{67}{72}$ & 0.78 &0.68     \\
4 & $\frac{1}{4}$ & $\frac{29}{32}$ & 0.06 & 0 && 
5 & $\frac{4}{5}$ & $\frac{23}{25}$ & 0.73 &0.62 &&
6 & 1 & $\frac{7}{6}$ & 1.13 & 0.48                \\
\end{longtable}
\end{center}
In every case, we find that $\Delta_0(r,\mu) \geq \Delta_1(r,\mu)$ as required.  We note that $\Delta_1(6,\frac{1}{3})> \frac{13}{18}$, so the proof does not apply to $\xi = (6,\frac{1}{3},\frac{13}{18})$.

When $\xi = (6,\frac{1}{3},\frac{13}{18})$, we put $\Xi = ((5,0,0),(6,\frac{1}{3},\frac{13}{18}),(1,2,6))$ and compute $\rho_\Xi^2=4$.  If $\rk(\theta')> 6$ then Proposition \ref{prop-highRank} gives $\rho_{\Theta}^2\leq \frac{13}{7}$, so $W(\Xi)$ is not nested in $W(\Theta)$ in this case either.
\end{proof}

The proof of the theorem also gives the following nonemptiness result.

\begin{corollary}
If hypothesis \ref{hyp-asymptotic} or \ref{hyp-smallRank} holds, then $W(\Xi)$ is nonempty.  If $\xi = (6,\frac{1}{3},\frac{13}{18})$, the wall corresponding to $\xi'=(5,0,0)$ is nonempty.
\end{corollary}

\subsection{The ample cone, large discriminant case}

Here we finish the proof that $W(\Xi)$ is the Gieseker wall if $\xi$ satisfies hypothesis \ref{hyp-asymptotic}.  
View $\xi = \xi(\Delta) = (\rk(\xi),\mu(\xi),\Delta)$ as having fixed rank and slope and variable $\Delta$, so that the extremal triple $\Xi=\Xi(\Delta)$ decomposing $\xi(\Delta)$ depends on $\Delta$.  We begin with the following lemma that will also be useful in the small rank case \ref{hyp-smallRank}.  

\begin{lemma}\label{lem-rightPoint}
The right endpoint $x^+_{\Xi(\Delta)} = s_{\Xi(\Delta)} + \rho_{\Xi(\Delta)}$ of $W(\Xi(\Delta))$ is a strictly increasing function of $\Delta$, and $$\lim_{\Delta\to \infty} x_{\Xi(\Delta)}^+ = \mu(\xi').$$
\end{lemma}
\begin{proof}
The statement that the function is increasing follows as in the second paragraph of Remark \ref{rem-extremalRemark}.  The walls $W(\Xi(\Delta))$ are all potential walls for the Chern character $\xi'$, so they form a nested family of semicircles foliating the quadrant left of the vertical wall $s = \mu(\xi')$.  If the radius of such a wall is arbitrarily large, then its right endpoint is arbitrarily close to the vertical wall.  We saw in the proof of Theorem \ref{thm-excludeHighRank} that if $\Delta$ is arbitrarily large, then the radius of $W(\Xi(\Delta))$ is arbitrarily large.
\end{proof}

\begin{theorem}\label{thm-main}
Suppose $\xi$ satisfies \ref{hyp-asymptotic}, and let $\Xi$ be the extremal triple decomposing $\xi$.  Then $W(\Xi) = W_{\max}$, and the primary edge of the ample cone of $M(\xi)$ corresponds to $W(\Xi)$.
\end{theorem}
\begin{proof}
Let $\Delta(\xi)$ be large enough that there is an extremal $\Xi$ that gives curves.  Also assume $\Delta(\xi)$ is large enough that Theorem \ref{thm-excludeHighRank} holds. If necessary, further increase $\Delta(\xi)$ so that no rational numbers with denominator at most $\rk(\xi)$ lie in the interval $[x^+_\Xi,\mu(\xi'))$.

Suppose the decomposition $\Theta$ corresponds to an actual wall $W(\Theta)$ which is at least as large as $W(\Xi)$, and let $$0\to F\to E \to Q\to 0$$ be a destabilizing sequence along $W(\Theta)$.  Since $F\in \cQ_s$ along $W(\Xi)$, we have $\mu(F) \geq x_\Xi^+$.  Now $\rk(F) \leq \rk(\xi)$, so by the choice of $\Delta(\xi)$ and the slope-closeness condition \ref{cond-slopeClose} we conclude $\mu(F) = \mu(\xi')$.

Furthermore, $F$ is $\mu$-semistable.  If it were not, by Lemma \ref{lem-boundsTrivial} the only possibility would be that $F$ has a subsheaf of slope $\mu(\xi)$.  Then $F$ must have a Harder-Narasimhan factor of slope smaller than $\mu(\xi')$, and this violates that $F\in \cQ_s$ for all $s$ along $W(\Theta)$ by our choice of $\Delta(\xi)$.

Finally, the $\mu$-semistability of $F$ implies $\Delta(F) \geq \Delta(\xi')$ by the discriminant minimality condition \ref{cond-discMinimal}. If $\Delta(F)>\Delta(\xi')$, then $W(\Theta)$ is nested inside $W(\Xi)$.  We conclude $\Delta(F) = \Delta(\xi')$, and $W(\Theta) = W(\Xi)$.  Therefore $W(\Xi)$ is the Gieseker wall.
\end{proof}

\begin{remark}\label{rem-explicit}
The lower bound on $\Delta$ needed for our proof of Theorem \ref{thm-main} can be made explicit.  We have increased $\Delta$ on several occasions throughout the paper.  If $r$ and $\mu$ are fixed, then $\Delta$ needs to be large enough that the following statements hold.
\begin{enumerate}
\item $\xi''$ is stable (Lemma \ref{lem-extremalExist}).
\item $\Xi$ gives curves.  Alternately, it is enough to know that if $F\in M(\xi')$ and $Q\in M(\xi'')$ are general, then $\Hom(F,Q)=0$.  The proof of Theorem \ref{thm-curves} allows us to give a lower bound for $\Delta$ if $\hom(F,Q)$ can be computed for some extremal triple $\Xi$ decomposing a character $\xi$ with rank $r$ and slope $\mu$.
\item The wall $W(\Xi)$ is large enough to imply the destabilizing subobject along $W_{\max}$ has rank at most $r$ (Proposition \ref{prop-highRank}).
\item The right endpoint $x_\Xi^+$ of $W(\Xi)$ is close enough to $\mu(\xi')$ that every rational number in $[x_{\Xi}^+,\mu(\xi'))$ has denominator larger than $r$.
\end{enumerate}
\end{remark}

\begin{remark}\label{rem-Yoshioka}
As an application of Theorem \ref{thm-main} and the discussion in the preceding remark, we explain how our results recover Yoshioka's computation \cite{Yoshioka} of the ample cone of $M(\xi)$ in case $c_1(\xi) = 1$ and $r\geq 2$.  Let $\theta' = (2,0,0)$ and $\theta = (r,\frac{1}{r},P(-\frac{1}{r})+\frac{1}{r})$, and let $\Theta$ be the corresponding admissible triple.  By Proposition \ref{prop-sporadic}, $\Theta$ is complete and gives curves.  Any triple $\Lambda \succeq \Theta$ also gives curves by Lemma \ref{lem-curveBump}. 

Now suppose $\xi$ has positive height and $c_1(\xi) = 1$.  Then either $\xi=\theta$ or $\xi$ is an elementary modification of $\theta$.  Let $\Xi$ be the extremal triple decomposing $\xi$.  If $W(\Xi)$ is the Gieseker wall, then the curves in $M(\xi)$ constructed in the previous paragraph are orthogonal to the divisor class on $M(\xi)$ coming from $W(\Xi)$, so we only need to check that $W(\Xi)$ is the Gieseker wall.  To do this, we verify that if $\Delta \geq P(-\frac{1}{r})+\frac{1}{r}$, then statements (1), (3), and (4) in Remark \ref{rem-explicit} hold.  It is clear that $\xi''$ is stable, so (1) holds.

To check (3) and (4), it is enough to verify they hold for the decomposition $\Theta$.  For (3), by Proposition \ref{prop-highRank} we must show $$\left(r-\frac{1}{2}\right)^2=\rho_\Theta^2 \geq \frac{r^2}{2(r+1)}\Delta(\theta)=\frac{2r^2-r+1}{4r+4},$$ which is clear for $r\geq 2$.  For (4), we need $x_\Theta^+>-\frac{1}{r}$; in fact,  $x_\Theta^+ = 0$ holds.
\end{remark}

\subsection{The ample cone, small rank case}  We next compute the Gieseker wall in the small rank case.

\begin{theorem}\label{thm-mainsmall}
Suppose $\xi$ satisfies \ref{hyp-smallRank}, and let $\Xi$ be the extremal triple decomposing $\xi$.  Then $W(\Xi) = W_{\max}$, and the primary edge of the ample cone of $M(\xi)$ corresponds to $W(\Xi)$.
\end{theorem}
\begin{proof}
We may assume $0 < \mu(\xi) \leq 1$.  Suppose $\Theta = (\theta',\theta,\theta'')$ is a decomposition of $\xi$ corresponding to an actual wall $W(\Theta)$ which is larger than $W(\Xi)$.  Let $F\to E$ be a destabilizing inclusion corresponding $W(\Theta)$.  We will show that $\mu(\theta')>\mu(\xi')$.  Combining this with   Lemma \ref{lem-boundsTrivial}, Theorem \ref{thm-excludeHighRank}, and slope-closeness \ref{cond-slopeClose} then gives a contradiction.
   
To prove $\mu(\theta')>\mu(\xi')$, we first derive two auxiliary inequalities. We will make use of the nondegenerate symmetric bilinear form $(\xi,\zeta) = \chi(\xi \te \zeta)$ on $K(\P^2)$.  Let $\gamma$ be a Chern character of positive rank such that $\gamma^\perp = \langle \xi',\xi\rangle$.

\emph{First inequality:} Since $\mu(\theta')<\mu(\xi)$, the assumption that $W(\Theta)$ is bigger than $W(\Xi)$ means $(\theta',\gamma) > 0$.  Indeed, $W(\Theta) = W(\Xi)$ if and only if $\theta'\in \gamma^\perp$.  If $\Delta(\theta')$ is decreased starting from a character on $\gamma^\perp$, then $(\theta',\gamma)$ increases and the wall $W(\Theta)$ gets bigger.

\emph{Second inequality:} Put $\zeta_1 = \ch \OO_{\P^2}(-1)$ and $\zeta_2 = \ch \OO_{\P^2}(-3).$  We observe that $\xi'$ lies in either $\zeta_1^\perp$ or $\zeta_2^\perp$.  Let $i$ be such that $\xi'\in \zeta_i^\perp$; we will show that $(\theta',\zeta_i)\leq 0$ in either case.

\emph{Case 1: $(\xi',\zeta_1)=0$.}  If $(\theta',\zeta_1)>0$, then $\chi(\OO_{\P^2}(1),F)>0$.  Suppose $\Ext^2(\OO_{\P^2}(1),F)=0$.  Then there is a nonzero homomorphism $\OO_{\P^2}(1)\to F$, and composing with the inclusion $F\to E$ gives a nonzero homomorphism $\OO_{\P^2}(1)\to E$.  Since $\xi$ has slope at most $1$ and positive height this contradicts semistability of $E$.

It remains to prove $\Ext^2(\OO_{\P^2}(1),F)=0$.  Dually, we must show $\Hom(F,\OO_{\P^2}(-2))=0$.  If $x^+_\Xi \geq -2$, then since $W(\Theta)$ is larger than $W(\Xi)$ we will have $F\in \cQ_{-2}$, proving this vanishing.  By Lemma \ref{lem-rightPoint}, we only have to check this inequality when $\Delta(\xi)$ is minimal subject to satisfying \ref{hyp-smallRank}.  We carried out this computation in Table \ref{table-discValues}.

\emph{Case 2: $(\xi',\zeta_2)=0$.}  If $(\theta',\zeta_2)>0$, then either $\Hom(\OO_{\P^2}(3),F)$ or $\Ext^2(\OO_{\P^2}(3),F)=\Hom(F,\OO_{\P^2})^*$ is nonzero.  Clearly $\Hom(\OO_{\P^2}(3),F)=0$ by Lemma \ref{lem-boundsTrivial}.  We must show $\Hom(F,\OO_{\P^2})=0$.  This follows from $x_{\Xi}^+\geq 0$, which is again true.

Now we use the inequalities $(\theta',\gamma)>0$ and $(\theta',\zeta_i)\leq 0$ to prove $\mu(\theta')>\mu(\xi')$.  There is a character $\nu$ such that if $\eta$ has positive rank, then $(\eta,\nu)\geq 0$ (resp. $>$) if and only if $\mu(\eta) \geq \mu(\xi')$ (resp. $>$).  We summarize the known information about the signs of various pairs of characters here, noting that $(\xi,\zeta_i)<0$ since $\xi$ has positive height.
$$
\begin{array}{c|ccc}
(-,-)& \gamma & \zeta_i & \nu\\
\hline \xi & 0 & <0 & >0 \\ 
\xi' & 0 & 0 & 0 \\
\theta' & >0 & \leq 0 & \\
\end{array}
$$ 
The character $\nu$ is in $(\xi')^\perp$, and $\gamma$ and $\zeta_i$ form a basis for $(\xi')^\perp$.  Write $\nu = a\gamma+b \zeta_i$ as a linear combination.  Since $0<(\xi,\nu) = b(\xi,\zeta_i)$, we find $b<0$.  The character $\nu$ has rank $0$, so this forces $a>0$.  We conclude $$(\theta',\nu) = a(\theta',\gamma)+b(\theta',\zeta_i)>0,$$ so $\mu(\theta')>\mu(\xi')$.
\end{proof}

We finish the paper by considering the last remaining case.

\begin{theorem}\label{thm-mainSporadic}
Let $\xi = (6,\frac{1}{3},\frac{13}{18})$, and let $\Xi$ be the decomposition of $\xi$ with $\xi' = (5,0,0)$.  Then $W(\Xi) = W_{\max}$, and the primary edge of the ample cone corresponds to $W(\Xi)$.
\end{theorem}
\begin{proof}
We use the same notation as in the proof of the previous theorem.  We compute $x_{\Xi}^+ = 0$, so since $W(\Theta)$ is larger than $W(\Xi)$ we have $F\in \cQ_\epsilon$ for some small $\epsilon>0$.  Consider the Harder-Narasimhan filtration $$0\subset F_1\subset \cdots \subset F_\ell = F.$$ Every quotient $\gr_i$ of this filtration satisfies $0<\mu(\gr_i)\leq \frac{1}{3}$.  Since $\gr_i$ is semistable, we find $\chi(\gr_i,\OO_{\P^2})\leq 0$ for all $i$.  It follows that $\chi(\theta',\ch(\OO_{\P^2}))\leq 0$ as well.  A straightforward computation using this inequality and $\chi(\theta',\gamma)>0$ shows $\mu(\theta')> \frac{2}{7}$.   This contradicts Theorem \ref{thm-excludeHighRank} since every rational number in the interval $(\frac{2}{7},\frac{1}{3})$ has denominator larger than $6$.
\end{proof}

\bibliographystyle{plain}

\begin{thebibliography}{ABCH}


\bibitem[AB]{ArcaraBertram}
D. Arcara, A. Bertram. Bridgeland-stable moduli spaces for $K$-trivial surfaces, with an appendix by Max Lieblich, J. Eur. Math. Soc., {\bf 15} (2013), 1--38. 

\bibitem[ABCH]{ABCH}
D. Arcara, A. Bertram, I. Coskun, and J. Huizenga.
\newblock The minimal model program for the Hilbert scheme of points on $\mathbb{P}^2$ and Bridgeland stability.  Adv. Math., {\bf 235} (2013), 580--626. 


\bibitem[BM]{BayerMacri}
A. Bayer and\ E. Macr\`i, The space of stability conditions on the local projective plane, Duke Math. J. {\bf 160} (2011), no.~2, 263--322.

\bibitem[BM2]{BayerMacri2}
A.~Bayer and E.~Macr\`{i}. Projectivity and birational geometry of Bridgeland
moduli spaces, J. Amer. Math. Soc., {\bf 27} (2014),  707--752.

\bibitem[BM3]{BayerMacri3}
A.~Bayer and E.~Macr\`{i}. MMP for the moduli spaces of sheaves on K3s via wall-crossing: nef and movable cones, Lagrangian fibrations, Invent. Math., to appear.

\bibitem[BC]{BertramCoskun}
A.~Bertram and I.~Coskun, The birational geometry of the Hilbert scheme of points on surfaces, in {\it Birational geometry, rational curves, and arithmetic},  Simons Symposia, Springer, New York, 2013, 15--55. 

\bibitem[Br1]{bridgeland:stable}
T. Bridgeland, Stability conditions on triangulated categories, Ann. of Math. (2) {\bf 166} (2007), no.~2, 317--345.

\bibitem[Br2]{Bridgeland}
T. Bridgeland, Stability conditions on $K3$ surfaces, Duke Math. J. {\bf 141} (2008), no.~2, 241--291.


\bibitem[CH]{CoskunHuizenga}
I.~Coskun and J.~Huizenga, Interpolation, Bridgeland stability and monomial schemes in the plane, J. Math. Pures Appl., to appear.

\bibitem[CHW]{CoskunHuizengaWoolf}
I.~Coskun, J.~Huizenga and M.~Woolf, The effective cone of the moduli spaces of sheaves on the plane, preprint. 

\bibitem[Dr]{DrezetBeilinson} J.-M. Drezet, Fibr\'es exceptionnels et suite spectrale de Beilinson g\'en\'eralis\'ee sur ${\bf P}\sb 2({\bf C})$, Math. Ann. {\bf 275} (1986), no.~1, 25--48. 

\bibitem[Dr2]{Drezet}
J.-M. Drezet, Fibr\'es exceptionnels et vari\'et\'es de modules de faisceaux semi-stables sur ${\bf P}\sb 2({\bf C})$, J. Reine Angew. Math. {\bf 380} (1987), 14--58. 



\bibitem[DLP]{DLP}
J.-M. Drezet\ and\ J. Le Potier, Fibr\'es stables et fibr\'es exceptionnels sur ${\bf P}\sb 2$, Ann. Sci. \'Ecole Norm. Sup. (4) {\bf 18} (1985), no.~2, 193--243. 











\bibitem[H]{HuizengaPaper2}
J.Huizenga, Effective divisors on the Hilbert scheme of points in the plane and interpolation for stable bundles,  J. Algebraic Geom., to appear.

\bibitem[HuL]{HuybrechtsLehn} D. Huybrechts\ and\ M. Lehn, {\it The geometry of moduli spaces of sheaves}, second edition, Cambridge Mathematical Library, Cambridge Univ. Press, Cambridge, 2010.

\bibitem[La]{Lazarsfeld} R. Lazarsfeld, {\it Positivity in algebraic geometry I, Classical setting: line bundles and linear series}, Springer-Verlag, 2004. 


\bibitem[LP]{LePotierLectures}
J. Le Potier, {\it Lectures on vector bundles}, translated by A. Maciocia, Cambridge Studies in Advanced Mathematics, 54, Cambridge Univ. Press, Cambridge, 1997. 

\bibitem[LZ]{LiZhao}
C. Li and X. Zhao, The MMP for deformations of Hilbert schemes of points on the projective plane, preprint.

\bibitem[Li1]{JunLiDonaldson}
J. Li, Algebraic geometric interpretation of Donaldson's polynomial invariants, J. Differential Geom. {\bf 37} (1993), no.~2, 417--466.

\bibitem[LQZ]{li}
W.-P. Li, Z. Qin\ and\ Q. Zhang, Curves in the Hilbert schemes of points on surfaces, in {\it Vector bundles and representation theory (Columbia, MO, 2002)}, 89--96, Contemp. Math., 322, Amer. Math. Soc., Providence, RI. 

\bibitem[MM]{MM}
A. Maciocia\ and\ C. Meachan, Rank 1 Bridgeland stable moduli spaces on a principally polarized abelian surface, Int. Math. Res. Not. {\bf 2013}, no.~9, 2054--2077.

\bibitem[MYY1]{MYY1}
H. Minamide, S. Yanagida, and K. Yoshioka, Fourier-Mukai transforms and the wall-
crossing behavior for Bridgeland's stability conditions, preprint.

\bibitem[MYY2]{MYY2}
H. Minamide, S. Yanagida, and K. Yoshioka, Some moduli spaces of Bridgeland’s stability
conditions, Int. Math. Res. Not., to appear.


\bibitem[N]{Nuer}
H. Nuer, Projectivity and birational geometry of moduli spaces of Bridgeland stable objects on an Enriques surface, preprint.


\bibitem[O]{Ohkawa}
R. Ohkawa, Moduli of Bridgeland semistable objects on ${\bf P}\sp 2$, Kodai Math. J. {\bf 33} (2010), no.~2, 329--366.

\bibitem[St]{Stromme}
S. A. Str\o mme, Ample divisors on fine moduli spaces on the projective plane, Math. Z. {\bf 187} (1984), no.~3, 405--423.


\bibitem[W]{Woolf}
M. Woolf, Effective and nef cones of moduli spaces of torsion sheaves on the projective plane, preprint.

\bibitem[YY]{YanagidaYoshioka}
S. Yanagida\ and\ K. Yoshioka, Bridgeland's stabilities on abelian surfaces, Math. Z. {\bf 276} (2014), no.~1-2, 571--610.

\bibitem[Y]{Yoshioka}
K. Yoshioka, A note on moduli of vector bundles on rational surfaces, J. Math. Kyoto Univ. {\bf 43} (2003), no.~1, 139--163. 

\bibitem[Y2]{Yoshioka2}
K. Yoshioka,  Bridgeland's stability and the positive cone of the moduli spaces of stable objects on an abelian surface, preprint. 


\end{thebibliography}

\end{document}